\documentclass{m2an}
\usepackage{framed,multirow}
\usepackage{latexsym}
\usepackage{amssymb,amsmath,mathtools}
\usepackage{graphicx, version, color}
\usepackage{subfigure}
\usepackage{xcolor}
\usepackage{subcaption}
\usepackage{caption}
\usepackage{cases}
\usepackage{exscale}
\usepackage[colorlinks=true, citecolor=blue, linkcolor=blue]{hyperref}
\newtheorem{theorem}{Theorem}[section]
\newtheorem{lemma}[theorem]{Lemma}
\newtheorem{proposition}[theorem]{Proposition}
\theoremstyle{definition}

\theoremstyle{remark}
\newtheorem{remark}[theorem]{Remark}
\numberwithin{equation}{section}
\usepackage{xcolor}
\def\dx{\hspace{0mm}{\mathrm{d}\boldsymbol{x}}}

\def\ds{\hspace{0mm}{\mathrm{d}s}}
\newcommand{\tF}{{\rm F}}
\newcommand{\pt}{{\partial_t}}
\newcommand{\ps}{{\partial_s}}
\newcommand{\mNk}{\mathcal{N}_{\kappa}}
\newcommand{\mL}{\mathcal{L}}
\newcommand{\mI}{\mathcal{I}}
\newcommand{\mX}{\mathcal{X}}
\newcommand{\bx}{\boldsymbol{x}}

\newcommand{\od}{\mathrm{d}}
\newcommand{\odt}{\mathrm{d}t}
\newcommand{\ods}{\mathrm{d}s}
\newcommand{\no}{\nonumber}
\newcommand{\la}{\langle }
\newcommand{\ra}{\rangle}

\usepackage{epsfig}
\graphicspath{{fig/}}

\begin{document}
\title{A generalized matrix-valued Allen--Cahn model and its numerical solution}
\author{Yaru Liu}\address{School of Mathematical Sciences, University of Electronic Science and Technology of China, Sichuan, China}
\author{Chaoyu Quan}\address{School of Science and Engineering, The Chinese University of Hong Kong (Shenzhen), Shenzhen, Guangdong, China \& Shenzhen International Center for Industrial and Applied Mathematics, Shenzhen Research Institute of Big Data, Guangdong, China}
\author{Dong Wang}\address{School of Science and Engineering, The Chinese University of Hong Kong (Shenzhen), Shenzhen, Guangdong, China \& Shenzhen International Center for Industrial and Applied Mathematics, Shenzhen Research Institute of Big Data, Guangdong, China \& Shenzhen Loop Area Institute, Guangdong 518048, China}
\begin{abstract} 
This paper introduces a generalized matrix-valued Allen--Cahn model, where the unknown matrix-valued field belongs to $\mathbb{R}^{m_1\times m_2}$ with dimension $m_1\geq m_2$. By taking different values of $m_1$ and $m_2$, this model covers the classical scalar-valued, vector-valued, and square-matrix-valued Allen--Cahn equations. At the continuous level, the proposed model is proven to admit a unique solution satisfying the maximum bound principle (MBP) and the energy dissipation law. At the discrete level, a class of arbitrarily high-order exponential time differencing Runge--Kutta (ETDRK) schemes is investigated that preserve the MBP unconditionally. Moreover, we prove that the first- and second-order ETDRK schemes satisfy the discrete energy dissipation unconditionally, while third- and higher-order schemes preserve the discrete energy dissipation under suitable time-step constraints. The proof of the sharp convergence order in time is provided. Numerical experiments are carried out to confirm our theoretical results. \end{abstract}

%
\maketitle
\section{Introduction}
Recently, matrix-valued Allen--Cahn models \cite{wang2019interface,osting2020diffusion} have attracted considerable attention to describe interface evolution in complex, multidimensional systems, and are closely related to diffuse-interface and matrix-valued field models arising in image processing problems \cite{batard2014covariant,rosman2014augmented}, crystal simulations \cite{elsey2013simple,elsey2014fast} and directional field synthesis problems \cite{vaxman2016directional}. They can be viewed as matrix-valued generalizations of the classical Allen--Cahn equation \cite{allen1979microscopic}, which serves as one of the fundamental diffuse-interface models for phase transitions and interface dynamics. Such matrix-valued extensions are essential for capturing higher-dimensional and more intricate interface behaviors.

In this paper, we propose a generalized matrix-valued Allen--Cahn model of the field $U(t,\bx)\in\mathbb{R}^{m_1\times m_2}$ with $m_1\geq m_2\geq 1$:
\begin{equation}\label{eq:MAC}
\begin{cases}
U_t(t,\bx)=\mathcal{L} U(t,\bx)+f(U(t,\bx)),\quad &(t,\bx)\in(0,T]\times\Omega,\\
U(0,\bx)=U_0(\bx),\quad &(t,\bx)\in\{t=0\}\times\overline{\Omega},
\end{cases}
\end{equation}
equipped with the periodic boundary condition. Here, $\mathcal{L}=\varepsilon^2\Delta$ is a diffusion operator scaled by $\varepsilon^2$, where $\Delta$ is the Laplacian operator, the parameter $\varepsilon>0$ represents the interfacial width, controlling the sharpness of the transition. $U_0$ is the initial condition, and $\Omega\subset\mathbb{R}^d$ ($d=1,2,3$) denotes an open, connected and bounded domain. The nonlinear term $f(U)$ is defined by
\begin{equation*}
f(U)=U-UU^\top U.
\end{equation*}
and $-f(U)$ is the variational derivative of $\int_\Omega \langle F(U), I_{m_2}\rangle_\tF$ with
\begin{equation*}
F(U)=\frac 14 (U^\top U-I_{m_2})^2,
\end{equation*}
and $I_{m_2}$ is the $m_2\times m_2$ identity matrix. The assumption $m_1\ge m_2$ guarantees that the target steady-state manifold $\mathcal{S}=\{U\in\mathbb R^{m_1\times m_2}:U^\top U=I_{m_2}\}$ is nonempty and the steady states are column-orthogonal. When $m_1<m_2$, one could consider $UU^\top=I_{m_1}$ as a restriction to ensure the row-orthogonality of steady states, where $I_{m_1}\in\mathbb{R}^{m_1\times m_1}$ is the identity matrix. In this paper, we restrict our analysis to the column-orthogonality case $m_1\geq m_2$.

The generalized matrix-valued model \eqref{eq:MAC} unifies the following three types of Allen--Cahn models that have been well studied.
\begin{itemize}
\item[(i)] {\em Scalar-valued case} ($m_1=m_2=1$): In this case, \eqref{eq:MAC} becomes $\partial_t u=\varepsilon^2\Delta u-F'(u)$, a  gradient flow of $E(u)=\int_\Omega \frac{\varepsilon^2}{2}|\nabla u|^2+F(u)\,d\bx$, approximating mean curvature flow in the sharp-interface limit. The classical Allen--Cahn equation \cite{allen1979microscopic} thus serves as a prototypical diffuse-interface model for curvature-driven interface motion and has been widely used in phase-field descriptions of two-phase incompressible flows \cite{yang2006numerical,gal2010longtime}.

\item[(ii)] {\em Vector-valued case} ($m_1\ge2,~m_2=1$): 
In \cite{rubinstein1989fast,rubinstein1989reaction}, Rubinstein et al. proposed a vector-valued reaction–diffusion system with Neumann boundary conditions for certain chemical reaction processes and, in this framework, studied the interface limits of the vector-valued Allen--Cahn equation (related with complex Ginzburg--Landau \cite{du1992analysis}) and its connections with geometric flows. Subsequently, Lin et al. \cite{lin2012reaction} systematically analyzed the static phase transition problem for this model in higher dimensions, while Liu\ \cite{Liu2025phase} extended these results to the dynamic case. Recently, multi-component conserved Allen--Cahn equations have been analyzed in \cite{grasselli2024multi}, where the sum of all components is prescribed to be a constant.

\item[(iii)] {\em Square-matrix-valued case} ($m_1=m_2\ge2$): The square-matrix-valued Allen--Cahn equation was first introduced by Osting and Wang in \cite{osting2020diffusion} to search for orthogonal matrix-valued harmonic mappings on periodic domains or closed surfaces. Subsequently, Wang et al. \cite{wang2019interface} systematically investigated the interfacial asymptotic dynamics of this equation from a multiple time-scale perspective. More recently, Fei et al. \cite{fei2023matrix} carried out a rigorous analysis of the sharp-interface limit of the matrix-valued Allen--Cahn equation.
\end{itemize}

At the continuous level, \eqref{eq:MAC} is the $L^2$ gradient flow of the following energy functional
\begin{equation}\label{eq:energy}
E(U)=\int_{\Omega}\Big(\frac{\varepsilon^2}{2}\|\nabla U\|_\tF^2+\la F(U),I_{m_2}\ra_\tF\Big)\dx,
\end{equation}
and satisfies the energy dissipation law (see Theorem \ref{theorem: energy dissipation of mac})
\begin{equation*}
\frac{\od}{\odt}E(U)=-\int_{\Omega}\|\pt U\|_\tF^2\dx\leq 0.
\end{equation*}
Here, $\|\cdot\|_\tF$ denotes the Frobenius norm and $\la\cdot,\cdot\ra_\tF$ represents the Frobenius inner product.
Furthermore, we prove that the proposed model \eqref{eq:MAC} admits a unique solution satisfying a maximum bound principle (MBP) (see Theorem \ref{thm:mbp}). That is to say, if the initial condition $U_0\in\mathbb{R}^{m_1\times m_2}$ of the generalized matrix-valued Allen--Cahn model satisfies
\begin{equation*}
\|U_0\|_\tF\leq \sqrt{m_2},
\end{equation*}
then the corresponding solution $U(t,\bx)$ fulfills
\begin{equation*}
\|U\|_\tF\leq \sqrt{m_2}.
\end{equation*}

At the discrete level, we focus on the structure preservation of time-discretizations for generalized matrix-valued Allen--Cahn equations. Note that some energy dissipative and MBP-preserving time-discretization schemes have been well studied for the square-matrix-valued Allen--Cahn equation, including the operator splitting schemes \cite{li2022stability2,quan2026unconditional}, the integrating factor Runge--Kutta schemes \cite{sun2023maximum}, and the exponential time differencing Runge--Kutta (ETDRK) schemes \cite{du2021maximum,li2021unconditionally,fu2022energy,liu2025maximum}. In particular, Liu et al. \cite{liu2025maximum} show that the maximum principle in Frobenius norm and the energy dissipation law of the first- and second-order ETDRK schemes are preserved unconditionally, even if the initial matrix field is nonsymmetric. Quan et al. \cite{quan2026unconditional} establish the unconditional energy dissipation of the second-order Strang splitting method for the square-matrix-valued Allen--Cahn equation. On higher-order methods, a class of arbitrarily high-order ETDRK schemes is proposed \cite{quan2025maximum}, preserving the MBP and the original energy dissipation law simultaneously for the scalar Allen--Cahn equation.

This paper first proves the well-posedness of the newly introduced matrix-valued Allen--Cahn model \eqref{eq:MAC} with $U\in\mathbb{R}^{m_1\times m_2}$, $m_1\geq m_2$, which covers the scalar-valued, vector-valued, and square-matrix-valued Allen--Cahn equations as special cases.
Furthermore, a rescaling technique \cite{quan2025maximum} is employed to construct arbitrarily high-order ETDRK schemes for the generalized matrix-valued Allen--Cahn equation, which preserve the MBP unconditionally and the original energy dissipation law under a suitable time-step constraint. In addition, a rigorous convergence analysis in time is provided.

The rest of this paper is organized as follows. Section~\ref{sec2} is devoted to the existence and uniqueness of solutions, the MBP, and the original energy dissipation law of the generalized matrix-valued Allen--Cahn equation. Section~\ref{sec3} introduces a class of arbitrarily high-order rescaled ETDRK schemes for the generalized matrix-valued Allen--Cahn equation and shows that they unconditionally preserve the MBP and satisfy the energy dissipation law under suitable time-step constraints, together with a detailed temporal convergence analysis. Section~\ref{sec4} presents numerical experiments that verify the theoretical results and demonstrate the effectiveness of the proposed schemes. Finally, some concluding remarks are given in Section~\ref{sec5}.

\section{Existence, uniqueness, maximum principle, and energy dissipation law of the solution}\label{sec2}
In this section, we investigate the existence and uniqueness of solutions to the generalized matrix-valued Allen--Cahn equation, as well as the MBP and energy dissipation law.

By introducing the linear elliptic operator $\mL_\kappa:=\varepsilon^2\Delta-\kappa\mI$ and nonlinear operator $\mNk[U]:=\kappa U+f(U)$, we can rewrite \eqref{eq:MAC} into the following form:
\begin{equation}\label{eq:MAC_ka}
U_t=\mL_{\kappa} U+\mNk[U], \quad (t,\bx)\in(0,T]\times\Omega,
\end{equation}
where $\mathcal I$ denotes the identity operator.

Then, we define a general Banach space $\mathcal{X}=C\left(\overline{\Omega};\mathbb{R}^{m_1\times m_2}\right)$ with $m_1 \geq m_2$, consisting of continuous $\mathbb{R}^{m_1\times m_2}$-valued functions on $\overline{\Omega}$, equipped with the supremum norm
\begin{equation*}
\|W\|_{\mathcal X}=\max _{\bx  \in \overline{\Omega}}\|W(\bx )\|_\tF, \quad \forall~ W \in \mathcal{X},
\end{equation*}
and introduce the closed subset
\begin{equation*}
\mathcal{X}_{m_2}=\{W(\bx )\in \mathcal X~|~ \|W(\bx )\|_{\mathcal{X}}\leq \sqrt{m_2}\}.
\end{equation*}

To prepare for the well-posedness and maximum principle analysis of \eqref{eq:MAC}, we first establish several basic stability properties of the nonlinear operator $\mNk$: $\mNk[U]$ is bounded and Lipschitz continuous on $\mathcal{X}_{m_2}$. 
Meanwhile, the semigroup $\{e^{t\varepsilon^2 \Delta}\}_{t\geq0}$ satisfies a contraction property on $\mathcal{X}$. 
These properties are collected in the lemmas below.

\begin{lemma}\label{Nkappa}
Let $m_1 \geq m_2$ be two positive integers. For any $U\in \mathbb{R}^{m_1\times m_2}$ with $\|U\|_\tF\leq \sqrt{m_2}$, if $\kappa\geq \max\{\frac 32 m_2-1,2\}$, then we have 
\begin{equation*}
\|\mNk[U]\|_\tF\leq\kappa\sqrt{m_2}.
\end{equation*}
Furthermore, for any $U_1,U_2\in \mathbb{R}^{m_1\times m_2}$ with $\|U_1\|_\tF,\|U_2\|_\tF\leq \sqrt{m_2}$, if $\kappa\geq 3m_2+1$, we have
\begin{equation*}
\|\mNk[U_1]-\mNk[U_2]\|_\tF\leq 2\kappa\|U_1-U_2\|_\tF.
\end{equation*}
\end{lemma}
\begin{proof}
For any $U\in\mathbb{R}^{m_1\times m_2}$, the singular value decomposition (SVD) of $U$ is written as $U=P\Sigma Q^\top$, where $P\in\mathbb{R}^{m_1\times m_1}$ and $Q\in\mathbb{R}^{m_2\times m_2}$ are orthogonal matrices, $\Sigma\in\mathbb{R}^{m_1\times m_2}$ is the diagonal matrix of singular values $\sigma_i$ with $i=1,\ldots,m_2$.
Thus, we have
\begin{equation*}
\begin{aligned}
\|\mNk[U]\|_\tF^2=&\|\kappa U+U-UU^\top U\|_\tF^2\\
=&\sum_{i=1}^{m_2}((\kappa+1)\sigma_i-\sigma_i^3)^2\\
=&\sum_{i=1}^{m_2}(\kappa+1)^2\sigma_i^2-2(\kappa+1)\sigma_i^4+\sigma_i^6.
\end{aligned}
\end{equation*}
Since $\|U\|_\tF\leq \sqrt{m_2}$, we have $\sum_{i=1}^{m_2}\sigma_i^2\leq m_2$ and $\frac{1}{m_2}\sum_{i=1}^{m_2}\sigma_i^2\leq 1$. Let $q:=\sum_{i=1}^{m_2}\sigma_i^2\in[0,m_2]$ and $g(q)=(\kappa+1)^2q-2(\kappa+1)q^2+q^3$. If $\kappa\geq \max\{\frac 32 m_2-1,2\}$, we can obtain
\begin{equation*}
g''(q)\leq 0,~\forall~q\in[0,m_2]~\text{and}~ g'(q)\geq 0,~\forall~q\in[0,1].
\end{equation*}
Then, we can use Jensen's inequality to obtain
\begin{equation*}
\sum_{i=1}^{m_2}\frac{1}{m_2}g(\sigma_i^2)\leq g\Big(\sum_{i=1}^{m_2}\frac{1}{m_2}\sigma_i^2\Big)\leq g(1)=\kappa^2.
\end{equation*}
Therefore $\|\mNk[U]\|_\tF\leq \kappa\sqrt{m_2}$, as claimed.

For $U_1,U_2\in\mathbb{R}^{m_1\times m_2}$ with $\|U_1\|_\tF,\|U_2\|_\tF\leq \sqrt{m_2}$, we have
\begin{equation*}
\begin{aligned}
&\|\mNk[U_1]-\mNk[U_2]\|_\tF\\
=&\|(\kappa+1)(U_1-U_2)-(U_1U_1^\top U_1-U_2U_2^\top U_2)\|_\tF\\
=&\|(\kappa+1)(U_1-U_2)-((U_1-U_2)U_1^\top U_1+U_2(U_1-U_2)^\top U_1+U_2U_2^\top (U_1-U_2))\|_\tF\\
\leq &(\kappa+1+\|U_1\|_\tF^2+\|U_2\|_\tF^2+\|U_1\|_\tF\|U_2\|_\tF)\|U_1-U_2\|_\tF\\
\leq&(\kappa+1+3m_2)\|U_1-U_2\|_\tF.
\end{aligned}
\end{equation*}
Therefore $\|\mNk[U_1]-\mNk[U_2]\|_\tF\leq 2\kappa\|U_1-U_2\|_\tF$ as $\kappa\geq 1+3m_2$.
\end{proof}

\begin{lemma}\label{lambda-Delta}
For any $U\in \mathcal X$, the solution operator of the Laplace equation with periodic boundary condition, i.e., $\{e^{t\varepsilon^2\Delta}\}_{t\geq0}$ on $\mathcal{X}$, satisfies the contraction property that $\forall~t\geq 0$, 
\begin{equation*}
\|e^{t\varepsilon^2 \Delta} U\|_{\mathcal X}\leq\|U\|_{\mathcal{X}}.
\end{equation*}
\end{lemma}
\begin{proof}
Since $C^2_{\rm{per}}\left(\overline{\Omega}; \mathbb{R}^{m_1\times m_2}\right)\coloneqq \{U\in C^2\left(\overline{\Omega}; \mathbb{R}^{m_1\times m_2}\right), U~\mbox{is periodic}\}$ is dense in $\mX$, it suffices to consider the case of $U\in C^2_{\mathrm{per}}$.

Let
\begin{equation}\label{eq:heat-solution}
W(t,\bx)=e^{t\varepsilon^2\Delta}U(\bx), \quad (t,\bx)\in(0,T]\times\Omega.
\end{equation}
be the classical solution to the heat equation with the periodic boundary condition:
\begin{equation}\label{eq:heat}
\begin{cases}
W_t(t,\bx)=\varepsilon^2\Delta W(t,\bx),\quad &(t,\bx)\in(0,T]\times\Omega,\\
W(0,\bx)=U(\bx),\quad &(t,\bx)\in\{t=0\}\times\overline{\Omega}.
\end{cases}
\end{equation}

Multiplying both sides of the first equation in \eqref{eq:heat} by $W^\top(t,\bx)$ and taking the trace, we have
\begin{align*}
\text{tr}(W^\top(t,\bx)\partial_t W(t,\bx))=\text{tr}(\varepsilon^2W^\top(t,\bx)\Delta W(t,\bx)),
\end{align*}
which yields
\begin{equation*}
\begin{aligned}
&\frac 12\partial_t\Big(\sum_{i=1}^{m_1}\sum_{j=1}^{m_2} W_{ij}^2(t,\bx)\Big)=\varepsilon^2\sum_{i=1}^{m_1}\sum_{j=1}^{m_2} W_{ij}(t,\bx)\Delta W_{ij}(t,\bx)\\
=&-\varepsilon^2\sum_{i=1}^{m_1}\sum_{j=1}^{m_2} |\nabla W_{ij}(t,\bx)|^2+\frac 12 \varepsilon^2\Delta\Big(\sum_{i=1}^{m_1}\sum_{j=1}^{m_2} W_{ij}^2(t,\bx )\Big)\\
\leq&\frac 12 \varepsilon^2\Delta\Big(\sum_{i=1}^{m_1}\sum_{j=1}^{m_2} W_{ij}^2(t,\bx)\Big).
\end{aligned}
\end{equation*}
So $\|W(t,\bx)\|_\tF$ is a subsolution of the heat operator. Taking  the supremum norm $\|\cdot\|_{\mathcal{X}}$ on both sides of  \eqref{eq:heat-solution} and using the maximum principle of the heat equation, we can obtain
\begin{equation*}
\|W(t)\|_\mX=\|e^{t\varepsilon^2\Delta}U\|_\mX\leq\|U\|_\mX, \quad t\in(0,T].
\end{equation*}
\end{proof}

\begin{theorem}[Existence, uniqueness and maximum principle]\label{thm:mbp}
Consider the generalized matrix-valued Allen--Cahn equation \eqref{eq:MAC} with the periodic boundary conditions and initial value $U_0\in \mathcal{X}_{m_2}$. Then the problem \eqref{eq:MAC} has a unique solution $U \in C([0,T];\mathcal{X}_{m_2})$, implying
\begin{equation*}
\|U(t,\bx)\|_\tF \leq \sqrt{m_2},
\end{equation*}
for all $(t,\bx)\in [0,T]\times\Omega$.
\end{theorem}
\begin{proof} 
We choose $\kappa\geq 3m_2+1$, so that the boundedness and Lipschitz estimates in Lemma \ref{Nkappa} hold for $\mNk[U]$. Fix $t_1>0$. For any $V\in C([0,t_1];\mathcal{X}_{m_2})$, let $W:=W(t,\bx )$ be a solution of the following linear problem:
\begin{equation}\label{eq:thm3.2.1}
\begin{aligned}
\begin{cases}
W_t+\kappa W=\mathcal{L} W+\mNk[V],&\quad t\in(0,t_1],~\bx\in\Omega,\\
W(\bx,0)=U_0(\bx),&\quad \bx\in\overline{\Omega},
\end{cases}
\end{aligned}
\end{equation}
with the periodic boundary condition. Setting $\mNk[V]=0$ and $U_0=0$ in \eqref{eq:thm3.2.1} will lead to $W(t)=e^{t(\mathcal{L}-\kappa)}U_0=0$ for each $t\in[0,t_1]$. Thus, $W$ is the unique solution defined on $[0,t_1]\times\Omega$. By the Duhamel principle, 
\begin{equation}\label{eq:solution}
\begin{aligned}
W(t,\bx )=e^{t(\mathcal{L}-\kappa)}U_0(\bx )+\int_{0}^{t}e^{(t-s) 
(\mathcal{L}-\kappa)}\mNk[V(s,\bx )]\ds,\quad \bx \in \Omega, t\in(0,t_1].
\end{aligned}
\end{equation}
Taking  the supremum norm $\|\cdot\|_{\mathcal{X}}$ on both sides of  \eqref{eq:solution} and using Lemmas \ref{Nkappa} and \ref{lambda-Delta}, we obtain
\begin{equation*}
\begin{aligned}
\|W(t)\|_{\mathcal{X}}&\leq e^{-\kappa t}\|U_0\|_{\mathcal{X}}+\int_{0}^{t}e^{-\kappa(t-s)}\|\mNk[V(s)]\|_{\mathcal{X}}\ds\\
&\leq e^{-\kappa t}\sqrt{m_2}+\int_{0}^{t}e^{-\kappa(t-s)}\kappa \sqrt{m_2}\ds\\
&=\sqrt{m_2},
\end{aligned}
\end{equation*}
for $t\in(0,t_1]$.

Next, we define a mapping $\mathcal{A}: C([0,t_1];\mathcal{X}_{m_2})\to C([0,t_1];\mathcal{X}_{m_2})$, $V\mapsto W$ according to \eqref{eq:solution}. For any $V_1,V_2\in C([0,t_1];\mathcal{X}_{m_2})$, we set $W_1=\mathcal{A}V_1$ and $W_2=\mathcal{A}V_2$. From \eqref{eq:solution} and Lemma \ref{Nkappa}, we have
\begin{equation*}
\begin{aligned}
\|W_1(t)-W_2(t)\|_{\mathcal{X}}
\leq&\int_{0}^{t}e^{-\kappa(t-s)}\|\mNk[V_1(s)]-\mNk[V_2(s)]\|_{\mathcal{X}}\ds\\
\leq& 2\kappa\int_{0}^{t}e^{-\kappa(t-s)}\| V_1(s)-V_2(s)\|_{\mathcal{X}}\ds\\
\leq& 2\kappa\left(\frac{1}{\kappa}-\frac{1}{\kappa}e^{-\kappa t_1}\right)\| V_1-V_2\|_{\mathcal{X}},\quad \forall~t\in[0,t_1].
\end{aligned}
\end{equation*}
Thus, for $t_1<\kappa^{-1}\ln2$ such that $2\kappa\left(\frac{1}{\kappa}-\frac{1}{\kappa}e^{-\kappa t_1}\right)<1$, $\mathcal{A}$ is a contraction. Since $\mathcal{X}_{m_2}$ is closed in $\mathcal{X}$, we know that $C([0,t_1];\mathcal{X}_{m_2})$ is complete with respect to the metric induced by the norm $\|\cdot\|_\mathcal{X}$. Then, by Banach's fixed point theorem, we get that $\mathcal{A}$ has a unique fixed point in $C([0,t_1];\mathcal{X}_{m_2})$, which is the unique solution to the generalized matrix-valued Allen--Cahn equation (\ref{eq:MAC_ka}) on $[0,t_1]$. 

Then, we extend this solution to $[0, T]$. The above contraction argument gives a unique solution on $[0,t_1]$ and $U(t_1)\in\mathcal X_{m_2}$. Therefore, we may restart the same fixed-point argument on $[t_1,2t_1]$ with initial value $U(t_1)$. Since the contraction constant depends only on $\kappa$ and $t_1$, the same argument applies on $[jt_1,(j+1)t_1]$ for $j=1,\ldots,J$. After finitely many steps, with $Jt_1\geq T$, these local solutions cover $[0,T]$. The uniqueness of each subinterval ensures that the solutions obtained in adjacent intervals coincide at the endpoints and thus patch together into a unique solution in $C([0,T];\mathcal X_{m_2})$.
\end{proof}

The energy dissipation law is a fundamental property of $L^2$ gradient flows, indicating that the energy of any sufficiently smooth solution to this equation decreases monotonically over time. For scalar cases, this conclusion can be directly proven by the chain rule of total derivatives. We extend this line of reasoning to the matrix-valued $L^2$ gradient flows \eqref{eq:energy} and establish the following energy dissipation law.
\begin{proposition}[Energy dissipation law]\label{theorem: energy dissipation of mac}
Let $U(t,\bx)\in\mathbb{R}^{m_1\times m_2}$ be a sufficiently smooth solution of
\eqref{eq:MAC} on $(0,T]\times\Omega$ with periodic boundary conditions.
Then the energy functional \eqref{eq:energy} satisfies
\begin{equation*}
\frac{\od}{\odt}E(U)\leq 0,\qquad t\in(0,T].
\end{equation*}
\end{proposition}
\begin{proof}
Differentiating \eqref{eq:energy} along a smooth solution and using integration by parts under the periodic boundary condition, we obtain
\begin{equation*}
\frac{\od}{\odt}E(U)=-\int_{\Omega}\left\langle \varepsilon^2\Delta U+f(U),\partial_t U\right\rangle_\tF\dx.
\end{equation*}
Since \eqref{eq:MAC} gives $\partial_t U=\varepsilon^2\Delta U+f(U)$, it follows that
\begin{equation*}
\frac{\od}{\odt}E(U)=-\int_{\Omega}\|\partial_t U\|_\tF^2\dx\leq0.
\end{equation*}
\end{proof}

\section{High-order rescaled ETDRK schemes}\label{sec3}
In this section, we develop arbitrarily high-order rescaled ETDRK schemes for the generalized matrix-valued Allen--Cahn equation that unconditionally preserve the MBP. Furthermore, we prove that the first- and second-order rescaled ETDRK schemes unconditionally satisfy a discrete energy dissipation law, while third- and higher-order schemes also obey a discrete energy dissipation law under suitable time-step constraints and achieve optimal error convergence in time.

The entire time interval $[0, T]$ is uniformly divided into $N$ sub-intervals, where $N$ is an arbitrary positive integer. Let the time step be $\tau=\frac TN$ and define $t_n=n\tau$ for $n=0,\ldots,N$. To establish the arbitrarily high-order schemes for solving \eqref{eq:MAC}, we focus on the equivalent equation \eqref{eq:MAC_ka} on the time interval $[t_n,t_{n+1}]$. Let $U^n(\bx)$ denote the numerical approximation to $U(t_n,\bx)$ at $t_n$. Given $U^n$, we define $W^n(s,\bx)$ as the exact solution of the following local initial value problem on $s\in[0,\tau]$:
\begin{equation}\label{eq:MAC-W}
\begin{aligned}
\begin{cases}
\ps W^n(s,\bx)=\mL_\kappa W^n(s,\bx)+\mNk[W^n(s,\bx)],&\quad (s,\bx)\in (0,\tau]\times\Omega,\\
W^n(0,\bx )=U^n(\bx ),&\quad \bx\in \overline{\Omega},
\end{cases}
\end{aligned}
\end{equation}
equipped with the periodic boundary condition. By Duhamel's principle, the solution of \eqref{eq:MAC-W} is given by
\begin{equation}\label{eq:mac-solution}
W^n(\tau,\bx )=e^{\tau\mL_\kappa}W^n(0,\bx)+\int_{0}^{\tau}e^{(\tau-s)\mL_\kappa}\mNk[W^n(s,\bx)]\ds.
\end{equation}

Then, we consider approximating the nonlinear terms in the integral to obtain the ETDRK$r$ scheme. For instance, by replacing $W^n(0,\bx )$ with the numerical solution $U^n$ and setting $\mNk[W^n(s,\bx)]\approx \mNk[U^n]$ in \eqref{eq:mac-solution} (here, it is assumed implicitly that the solution is sufficiently smooth in time), the ETDRK$1$ scheme reads
\begin{equation}\label{eq:ETDRK1}
\begin{aligned}
U^{n+1}=W^n(\tau)=e^{\tau\mL_\kappa}U^n+\int_{0}^{\tau}e^{(\tau-s)\mL_\kappa}\mNk[U^n]\ds,
\end{aligned}
\end{equation}
Based on the idea of the ETDRK1 scheme, higher-order ETDRK schemes can be derived by approximating $\mNk[W^n(s,\bx)]$ with higher-degree interpolation polynomials for $s\in[0,\tau]$. For any positive integer $r$, we can choose $r+1$ nodes $\{0=a_{r,0}<a_{r,1}<\ldots<a_{r,r}=1\}$ in the interval $[0,1]$ to construct, for each $\bx\in\Omega$, an $r$-th degree interpolating polynomial $P_r^n(s,\bx)$ such that $P_r^n(a_{r,k}\tau,\bx)=\mNk[W_r^n(a_{r,k}\tau,\bx)]$ for $1\leq k\leq r$, where $W_r^n(a_{r,k}\tau,\bx)$ are computed by the ETDRK$r$ scheme at nodes $\{a_{r,k}\}_{k=1}^r$ and $W_{r}^n(0,\bx)=U^n(\bx)$. Then, we can use $P_r^n(s,\bx)$ to approximate $\mNk[W^n(s,\bx)]$, thereby obtaining the $(r+1)$th-order ETDRK scheme. To facilitate the discussion in the following sections, we rewrite $P_r^n(s,\bx)$ in the following form:
\begin{equation}\label{eq:poly_P}
P_r^n(s,\bx)=\mNk[U^n](\bx)+C_{r,1}^n(\bx)\frac{s}{\tau}+C_{r,2}^n(\bx)\Big(\frac{s}{\tau}\Big)^2+\ldots+C_{r,r}^n(\bx)\Big(\frac{s}{\tau}\Big)^r,
\end{equation}
where the coefficients $\mathbf{C}_r^n(\bx)=[C_{r,1}^n(\bx);\cdots;C_{r,r}^n(\bx)]_{k=1}^r\in\mathbb{R}^{rm_1\times m_2}$ with $C_{r,k}^n(\bx)\in \mathbb{R}^{m_1\times m_2}$ for $k=1,\ldots,r$ are determined by
\begin{equation}\label{matrix:C}
\mathbf{C}_r^n(\bx)=\left(V_r^{-1}\otimes I_{m_1}\right)D_r^n(\bx),
\end{equation}
and
\begin{equation}\label{eq:matrix-v-d}
V_r=\begin{pmatrix}
a_{r,1}&(a_{r,1})^2&\ldots&(a_{r,1})^r\\
a_{r,2}&(a_{r,2})^2&\ldots&(a_{r,2})^r\\
\vdots&\vdots&\ddots&\vdots\\
a_{r,r}&(a_{r,r})^2&\ldots&(a_{r,r})^r
\end{pmatrix},
\quad
D_r^n(\bx)=\begin{pmatrix}
\mNk[W_r^n(a_{r,1}\tau)]-\mNk[U^n]\\
\mNk[W_r^n(a_{r,2}\tau)]-\mNk[U^n]\\
\vdots\\
\mNk[W_r^n(a_{r,r}\tau)]-\mNk[U^n]
\end{pmatrix}. 
\end{equation}
Moreover, the interpolation polynomial $P_r^n(s,\bx)$ can be expressed as
\begin{equation*}
P_r^n(s,\bx)=\sum_{k=0}^r\ell_{r,k}(s)\mNk[W_r^n(s_k,\bx)],
\end{equation*}
where $\ell_{r,k}(s)$ is the Lagrange basis function associated with the uniform interpolation node $s_k=k\tau/r$ for $k=0,\cdots,r$. In the following content, for the sake of readability, the spatial variable $\bx$ is sometimes omitted when its meaning is clear.

Similar to the proof in \cite{liu2025maximum}, we can obtain that the ETDRK$1$ and ETDRK$2$ schemes preserve the MBP, which implies that, for $r=0,1$, the interpolant $P_r^n(s,\bx)$ satisfies $\|P_r^n(s,\bx)\|_F\le\kappa\sqrt{m_2}$
for all $s\in[0,\tau]$ and $\bx\in\Omega$. However, for higher-order ETDRK schemes, we cannot ensure that the polynomial $P_{r}^n(s,\bx)$ used to approximate the nonlinear terms in \eqref{eq:mac-solution} satisfies $\|P_{r}^n(s,\bx)\|_\tF\leq \kappa\sqrt{m_2}$ for any $s\in[0,\tau]$ and $\bx\in \Omega$. Therefore, to establish an arbitrary higher-order ETDRK scheme that preserves the maximum principle, we define a scaling factor that depends on $\bx$:
\begin{equation*}
\alpha_{r}^n(\bx):=\min\left\{\frac{\kappa\sqrt{m_2}}{\max\limits_{s\in[0,\tau]}\|P_r^n(s,\bx)\|_\tF},1\right\}.
\end{equation*}
Let
\begin{equation}\label{widetilde-P}
\widetilde{P}_{r}^n(s,\bx)=\alpha_{r}^n(\bx) P_{r}^n(s,\bx).
\end{equation}
As a consequence, $\widetilde{P}_{r}^n(s,\bx)$ satisfies 
\begin{equation*}
\|\widetilde{P}_{r}^n(s,\bx)\|_\tF\leq \kappa \sqrt{m_2},
\end{equation*}
for $(s,\bx)\in (0,\tau]\times\Omega$.
We emphasize that the scaling factor is introduced only to enforce the discrete maximum bound and does not change the temporal order of the scheme. In the local consistency setting, $P_r^n(s,\bx)$ approximates $\mNk[U(t_n+s,\bx)]$ with an error of order $\mathcal O(\tau^{r+1})$. Since $\|\mNk[U(t_n+s,\bx)]\|_{\rm F}\leq \kappa\sqrt{m_2}$, the excess of $P_r^n$ over this bound is also of order $\mathcal O(\tau^{r+1})$. Hence, when the rescaling is active, $1-\alpha_r^n(\bx)=\mathcal O(\tau^{r+1})$, and the correction $\widetilde P_r^n-P_r^n$ is of sufficiently high order. A rigorous estimate that includes the numerical stage errors will be given later in Lemma \ref{lem:P-P} and will be used in the convergence proof.

Taking $\mNk[U(t_n+s,\bx)]\approx \widetilde{P}_{r}^n(s,\bx)$ in \eqref{eq:mac-solution} gives the rescaled ETDRK$(r+1)$ method:
\begin{equation}\label{eq:RETDRKr}
\begin{aligned}
U^{n+1}=&W_{r+1}^n=e^{\tau\mL_\kappa}U^n+\int_{0}^{\tau}e^{(\tau-s)\mL_\kappa}\widetilde{P}_{r}^n(s)\ds\\
=&e^{\tau\mL_\kappa}U^n+\int_{0}^{\tau}e^{(\tau-s)\mL_\kappa}\left(\alpha_{r}^n P_{r}^n(s)\right)\ds\\
=&e^{\tau\mL_\kappa}U^n+\tau\phi_1(\tau\mL_\kappa)\left(\alpha_r^n\mNk[U^n]\right)+\tau\sum_{k=1}^rk!\phi_{k+1}(\tau \mL_{\kappa})\left(\alpha_r^nC_{r,k}^n\right)\\
=&W_1^n(\tau)+\tau\phi_1(\tau\mL_\kappa)\left((\alpha_r^n-1)\mNk[U^n]\right)+\tau\sum_{k=1}^rk!\phi_{k+1}(\tau \mL_{\kappa})\left(\alpha_r^nC_{r,k}^n\right),
\end{aligned}
\end{equation}
where $W_1^n=e^{\tau \mL_\kappa}U^n+(e^{\tau \mL_\kappa}-\mI)\mL_\kappa^{-1}\mNk[U^n]$ is the numerical solution obtained by the ETDRK1 method approximating the exact solution of \eqref{eq:MAC} at $t=t_{n+1}$; $U^n$ is the numerical solution obtained by the rescaled ETDRK$(r+1)$ method approximating the exact solution of \eqref{eq:MAC} at $t=t_n$; $\phi_{k+1}(z):=\Big(e^z-\sum_{j=0}^k \frac{z^j}{j!}\Big)z^{-(k+1)}$ and $0\leq k\leq r$ is an integer.

\subsection{Maximum bound principles}
\begin{theorem}[Discrete MBP]\label{thm:MBP}
If $\|U_0\|_\tF\leq \sqrt{m_2}$ and $\kappa\geq \max\{\frac 32 m_2-1,2\}$, the rescaled ETDRK$r$ scheme \eqref{eq:RETDRKr} preserves the discrete MBP unconditionally, i.e., for any time step size $\tau>0$ and $n\geq 1$, the numerical solution $U^n$ produced by the rescaled ETDRK$r$ scheme satisfies
\begin{equation*}
\|U^n\|_\tF\leq \sqrt{m_2},
\end{equation*}
{for all $\bx\in\Omega$.}
\end{theorem}
\begin{proof}
Suppose that $\|U^n\|_\tF\leq\sqrt{m_2}$ holds for some $n\geq0$. We need to prove $\|U^{n+1}\|_\tF\leq\sqrt{m_2}$. First, assume that $\max\limits_{s\in[0,\tau]}\|P_{r-1}^n(s,\bx)\|_\tF\leq \kappa\sqrt{m_2}$. Then $\alpha_{r-1}^n(\bx)=1$ and
\begin{equation}\label{eq:thm3-1}
\|\widetilde{P}_{r-1}^n(s,\bx)\|_\tF=\|P_{r-1}^n(s,\bx)\|_\tF\leq\kappa\sqrt{m_2}, \quad {\forall~(s,\bx)\in [0,\tau]\times\Omega}.
\end{equation}
If $\max\limits_{s\in[0,\tau]}\|P_{r-1}^n(s,\bx)\|_\tF>\kappa\sqrt{m_2}$, then $\alpha_{r-1}^n(\bx)=\frac{\kappa\sqrt{m_2}}{\max\limits_{s\in[0,\tau]}\|P_{r-1}^n(s,\bx)\|_\tF}$ and
\begin{equation}\label{eq:thm3-2}
\|\widetilde{P}_{r-1}^n(s,\bx)\|_\tF=\frac{\kappa\sqrt{m_2}}{\max\limits_{s\in[0,\tau]}\|P_{r-1}^n(s,\bx)\|_\tF}\|P_{r-1}^n(s,\bx)\|_\tF\leq\kappa\sqrt{m_2},
\end{equation}
for any {$(s,\bx)\in [0,\tau]\times\Omega$}. According to \eqref{eq:RETDRKr}, we have the explicit expression for the rescaled ETDRK$r$ scheme:
\begin{equation}\label{eq:thm3-3}
U^{n+1}=e^{\tau\mL_\kappa}U^n+\int_{0}^{\tau}e^{(\tau-s)\mL_\kappa}\widetilde{P}_{r-1}^n(s)\ds,
\end{equation}
where $\mL_\kappa=\varepsilon^2\Delta-\kappa\mI$. Taking the supremum norm $\|\cdot\|_{\mathcal{X}}$ on both sides of \eqref{eq:thm3-3} and using \eqref{eq:thm3-1}-\eqref{eq:thm3-2} and Lemma \ref{lambda-Delta}, we obtain
\begin{equation*}
\begin{aligned}
\|U^{n+1}\|_{\mathcal{X}}=&\Big\| e^{\mL_\kappa\tau}U^n+\int_{0}^{\tau}e^{\mL_\kappa(\tau-s)}\widetilde{P}_{r-1}^n(s)\ds\Big\|_{\mathcal{X}}\\
\leq&e^{-\kappa\tau}\|e^{\tau\varepsilon^2\Delta}U^n\|_\mX+\int_{0}^{\tau}e^{-\kappa(\tau-s)}\|e^{\varepsilon^2\Delta(\tau-s)}\widetilde{P}_{r-1}^n(s)\|_{\mathcal{X}}\ds\\
\leq& e^{-\kappa\tau}\|U^n\|_\mX+\int_{0}^{\tau}e^{-\kappa(\tau-s)}\|\widetilde{P}_{r-1}^n(s)\|_{\mathcal{X}}\ds\\
\leq&\Big(e^{-\kappa\tau}+\left(\frac{1}{\kappa}-\frac{1}{\kappa}e^{-\kappa\tau}\right)\kappa\Big)\sqrt{m_2}=\sqrt{m_2},
\end{aligned}
\end{equation*}
which verifies the MBP of the rescaled ETDRK$r$ scheme.
\end{proof}

\subsection{Original energy dissipation}
This subsection proves that first- and second-order ETDRK$r$ schemes unconditionally satisfy the original energy dissipation law, while third- and higher-order schemes satisfy it under suitable time-step constraints. To ensure the completeness of the proof, we first present several preliminary lemmas, which serve as the primary tools for the subsequent demonstration.

\begin{lemma}\label{lem:martix_C}
Let $\mathbf{C}_r^n$ and $V_r$ be defined as in \eqref{matrix:C}–\eqref{eq:matrix-v-d}. Assume that $\kappa\geq 3m_2+1$, $\|U^n\|_\tF\leq \sqrt{m_2}$, and $\|W_r^n(a_{r,k}\tau)\|_\tF\leq\sqrt{m_2}$ for $1\leq k\leq r$. 
Then the following estimate holds:
\begin{equation}\label{eq:lem3.2-0}
\|\mathbf{C}_r^n\|_\tF
\leq \frac{2\kappa}{\sigma_{\min}(V_r)}
\left(\sum_{k=1}^r\|W_r^n(a_{r,k}\tau)-U^n\|_\tF^2\right)^{\frac12},
\end{equation}
where $\sigma_{\min}(V_r)$ denotes the smallest singular value of $V_r$.
\end{lemma}
\begin{proof}
For the matrix $V_r$ defined by \eqref{eq:matrix-v-d}, the SVD of $V_r^{-1}$ is written as $V_r^{-1}=P\Sigma Q^\top$, where $P$ and $Q$ are orthogonal matrices, and $\Sigma$ is the diagonal matrix of singular values $\sigma_i$ with $i=1,\ldots,r$. Thus, according to $\mathbf{C}_r^n=\left(V_r^{-1}\otimes I_{m_1}\right)D_r^n$, we have
\begin{equation}\label{eq:lem3.2-1}
\begin{split}
\|\mathbf{C}_r^n\|_\tF&=\Big(\sum_{j=1}^{m_2}\|\left(V_r^{-1}\otimes I_{m_1}\right)D_r^n[:,j]\|_2^2\Big)^{\frac 12}\\
&\leq\Big(\|\left(V_r^{-1}\otimes I_{m_1}\right)\|_2^2\sum_{j=1}^{m_2}\|D_r^n[:,j]\|_2^2\Big)^{\frac 12}\\
&=\|V_r^{-1}\otimes I_{m_1}\|_2\|D_r^n\|_\tF=\|V_r^{-1}\|_2\|D_r^n\|_\tF\\
&=\sigma_{\max}(V_r^{-1})\|D_r^n\|_\tF=\sigma_{\min}^{-1}(V_r)\|D_r^n\|_\tF,
\end{split}
\end{equation}
where $D_r^n[:,j]$ denotes the $j$-th column of $D_r^n$ and $\sigma_{\max}(V_r^{-1})=\sigma_{\min}^{-1}(V_r)=\max_{i=1,\ldots,r}|\sigma_i|$.

Based on the expression for the matrix $D_r^n$ in \eqref{eq:matrix-v-d} and using Lemma \ref{Nkappa}, we have
\begin{equation}\label{eq:lem3.2-2}
\|D_r^n\|_\tF=\left(\sum_{k=1}^r\|\mNk[W_r^n(a_{r,k}\tau)]-\mNk[U^n]\|_\tF^2\right)^{\frac 12}\leq 2\kappa\left(\sum_{k=1}^r\|W_r^n(a_{r,k}\tau)-U^n\|_\tF^2\right)^{\frac 12}.
\end{equation}
Combining \eqref{eq:lem3.2-1}-\eqref{eq:lem3.2-2} yields \eqref{eq:lem3.2-0}.
\end{proof}

To simplify the notations, we define the following $L^2$-norm for any $V=(V_{i,j})_{m_1\times m_2}\in\mathbb{R}^{m_1\times m_2}$ with $V_{i,j}\in L^2(\Omega)$:
$$\|V\|_{L^2(\Omega)}=\left(\int_\Omega\|V\|_\tF^2\dx\right)^\frac{1}{2}=\left(\sum_{i=1}^{m_1}\sum_{j=1}^{m_2}\|V_{ij}\|_{L^2(\Omega)}^2\right)^\frac{1}{2}.$$

To derive $L^2$ stability estimates for the operator $\phi_k(t\mL_\kappa)$ acting on matrix-valued functions, we state the following lemma, which is a direct extension of the scalar result in \cite{quan2025maximum} to the matrix case.

\begin{lemma}\label{lem:phi}
Let $t>0$, $k\in\mathbb{N}$ with $k\ge1$, and $V=(V_{i,j})_{m_1\times m_2}\in \mathbb{R}^{m_1\times m_2}$ with $V_{i,j}\in L^2(\Omega)$. Then, for any $\lambda\in(0,1)$, the formula $\phi_{k}(z):=\Big(e^z-\sum_{j=0}^{k-1} \frac{z^j}{j!}\Big)z^{-k}$ about non-positive operator $t\mL_\kappa$ satisfies the following inequality:
\begin{align}\label{eq:lem3.3-0}
\|\lambda^k\phi_k(\lambda t\mL_\kappa)V\|_{L^2(\Omega)}\leq \|\phi_k(t\mL_\kappa)V\|_{L^2(\Omega)}\leq\frac{1}{k!}\|V\|_{L^2(\Omega)}.
\end{align}
\end{lemma}
\begin{proof}
Since $\phi_k(t\mL_\kappa)$ and $\lambda^k\phi_k(\lambda t\mL_\kappa)$ act componentwise on matrix-valued functions, the assertion follows by applying the corresponding scalar estimate in \cite[Lemma~1]{quan2025maximum} to each component $V_{i,j}$:
\begin{equation*}
\|\lambda^k\phi_k(\lambda t\mL_\kappa)V_{i,j}\|_{L^2(\Omega)}
\leq
\|\phi_k(t\mL_\kappa)V_{i,j}\|_{L^2(\Omega)}
\leq
\frac{1}{k!}\|V_{i,j}\|_{L^2(\Omega)}.
\end{equation*}
Squaring these inequalities, summing over $i=1,\ldots,m_1$ and $j=1,\ldots,m_2$, and then taking square roots gives \eqref{eq:lem3.3-0}.
\end{proof}

\begin{lemma}\label{Lipschitz-f}
For any $U,V\in\mathbb{R}^{m_1\times m_2}$ with $m_1\geq m_2$, assume $\|U\|_\tF,\|V\|_\tF\leq \sqrt{m_2}$, if $\kappa\geq 3m_2+1$, we have
\begin{equation}\label{eq:lem4-0}
\int_\Omega\la F(U)-F(V),I_{m_2}\ra_\tF+\la U-V,f(V) \ra_\tF\dx\leq\frac{\kappa}{2}\int_\Omega\|U-V\|_\tF^2\dx.
\end{equation}
\end{lemma}
\begin{proof}
According to the expressions for $f$ and using $\|U\|_\tF,\|V\|_\tF\leq \sqrt{m_2}$ and $\kappa\geq 3m_2+1$, we can obtain
\begin{equation}\label{eq:lem4-1}
\begin{aligned}
\|f(V)-f(U)\|_\tF=&\|UU^\top U-U-(VV^\top V-V)\|_\tF\\
=&\|-(U-V)+UU^\top(U-V)+U(U-V)^\top V+(U-V)V^\top V\|_\tF\\
\leq&(1+\|U\|_\tF^2+\|U\|_\tF\|V\|_\tF+\|V\|_\tF^2)\|U-V\|_\tF\\
\leq&\kappa\|U-V\|_\tF.
\end{aligned}
\end{equation}
Let $W(s):=V+s(U-V)$ for $s\in[0,1]$ and define 
\begin{equation*}
g(s)=\int_\Omega\la F(W(s)),I_{m_2}\ra_\tF\dx.
\end{equation*}
Using the chain rule and \eqref{eq:lem4-1} yields
\begin{equation}\label{eq:lem4-2}
\begin{aligned}
g'(s)=&\int_\Omega\la -f(W(s)),U-V\ra_\tF\dx\\
=&\int_\Omega\la -f(V),U-V\ra_\tF+\la f(V)-f(W(s)),U-V\ra_\tF\dx\\
\leq&\int_\Omega\la -f(V),U-V\ra_\tF+\|f(V)-f(W(s))\|_\tF\|U-V\|_\tF\dx\\
\leq&\int_\Omega\la -f(V),U-V\ra_\tF+\kappa s\|U-V\|_\tF^2\dx,
\end{aligned}
\end{equation}
where we used $\|W(s)-V\|_\tF=s\|U-V\|_\tF$.

Integrating \eqref{eq:lem4-2} from $s=0$ to $1$ yields
\begin{equation}\label{eq:lem4-3}
\begin{aligned}
\int_0^1 g'(s)\ds=&g(1)-g(0)=\int_\Omega\la F(U)-F(V),I_{m_2}\ra_\tF\dx\\
\leq&\int_\Omega\int_0^1\Big(\la -f(V),U-V\ra_\tF+\kappa s\|U-V\|_\tF^2\Big)\ds\dx\\
=&\int_\Omega\la -f(V),U-V\ra_\tF+\frac{\kappa}{2}\|U-V\|_\tF^2\dx.
\end{aligned}
\end{equation}
Finally, rearrange \eqref{eq:lem4-3} to obtain \eqref{eq:lem4-0}.
\end{proof}

\begin{theorem}[Energy dissipation]\label{thm:energy}
Consider the generalized matrix-valued Allen--Cahn equation with periodic boundary conditions and assume that $\kappa\geq 3m_2+1$. 
Let $\{U^n\}_{n\ge0}$ be the numerical solution generated by the rescaled 
ETDRK$r$ scheme with time step $\tau>0$. Then for any $r\ge1$ there exists 
a constant $\tau_{\max,r}>0$, independent of $\varepsilon$, such that
\begin{equation*}
E(U^{n+1})-E(U^{n})\leq 0, \quad \text{for all } n\geq0,
\quad\text{whenever } \tau\le \tau_{\max,r}.
\end{equation*}
Moreover, $\tau_{\max,1}=\tau_{\max,2}=+\infty$, and for $r\geq3$,
\begin{equation*}
\tau_{\max,r}=\frac{1}{10\kappa}\min_{1\le j\le r-1}\frac{\sigma_{\min}(V_j)}{j},
\end{equation*}
where $\sigma_{\min}(V_j)$ denotes the smallest singular value of the matrix 
$V_j$ defined in \eqref{eq:matrix-v-d}. In particular, the scheme is unconditionally 
energy dissipative for $r=1$ and $2$, and conditionally energy dissipative  
for $r\geq3$ under the above step-size constraint.
\end{theorem}

\begin{proof}
Using the proof strategy for the square-matrix-valued Allen--Cahn equation with $m_1 = m_2$ in \cite{liu2025maximum}, we can show that, for the generalized matrix-valued Allen--Cahn equation, the ETDRK$1$ and ETDRK$2$ schemes also satisfy energy dissipation law without any restriction on $\tau$, that is, $\tau_{\max,1} = \tau_{\max,2} = +\infty$. The detailed proof for $r=1$, $2$ is omitted here, and we only focus on the rescaled ETDRK$r$ schemes with $r\geq 3$.

According to \eqref{eq:MAC_ka} and the definition of the energy in \eqref{eq:energy}, the energy of the function $W_r^n$ satisfies
\begin{equation}\label{eq:thm3.5-1}
\begin{aligned}
\frac{\od}{\ods}E(W_r^n)&=-\int_{\Omega}\la \varepsilon^2\Delta W_r^n+f(W_r^n),\ps W_r^n\ra_\tF\dx\\
&=\int_{\Omega}\la-\mL_{\kappa}W_r^n-\mNk [W_r^n],\ps W_r^n\ra_\tF\dx.
\end{aligned}
\end{equation}
Then, we apply the integral operator $\int_0^{\tau}\cdot~\ds$, where `$\cdot$' stands for any admissible integrand, to \eqref{eq:thm3.5-1} and obtain
\begin{equation}\label{eq:thm3.5-2}
E(U^{n+1})-E(U^n)=-\int_{\Omega}\Big\la\int_0^{\tau}\mL_{\kappa}W_r^n(\ps W_r^n)^{\top}\ds,I_{m_1}\Big\ra_\tF\dx-\int_{\Omega}\Big\la \int_0^{\tau}\mNk [W_r^n](\ps W_r^n)^{\top}\ds,I_{m_1}\Big\ra_\tF\dx.
\end{equation}
Applying integration by parts to the second term in \eqref{eq:thm3.5-2} and using Lemma \ref{Lipschitz-f} yields
\begin{equation}\label{eq:thm3.5-3}
\begin{split}
&\int_{\Omega}\left\la \int_0^{\tau}\mNk [W_r^n](\ps W_r^n)^{\top}\ds,I_{m_1}\right\ra_\tF\dx\\
=&\int_{\Omega}-\left\la F(U^{n+1})-F(U^n),I_{m_2}\right\ra_\tF+\frac{\kappa}{2}\left\la U^{n+1}(U^{n+1})^{\top}-U^n(U^n)^{\top},I_{m_1}\right\ra_\tF\dx\\
\geq&\int_{\Omega}\la f(U^n)(U^{n+1}-U^n)^{\top}+\kappa U^n(U^{n+1}-U^n)^{\top},I_{m_1}\ra_\tF\dx\\
=&\int_{\Omega}\la \mNk[U^n],U^{n+1}-U^n\ra_\tF\dx.
\end{split}
\end{equation}
Substituting \eqref{eq:thm3.5-3} into \eqref{eq:thm3.5-2}, employing the differential form 
\begin{equation*}
\ps W_r^n=\mL_\kappa W_r^n+\widetilde{P}_{r-1}^n(s), \quad s\in (0,\tau],~\bx \in \Omega.
\end{equation*}
and using the Cauchy–Schwarz inequality, we obtain
\begin{equation}\label{eq:thm3.5-5}
\begin{aligned}
E(U^{n+1})-E(U^n)\leq&\int_\Omega\Big\la \int_0^\tau(-\ps W_r^n+\widetilde{P}_{r-1}^n (s)-\mNk [U^n])(\ps W_r^n)^\top\ds,I_{m_1}\Big\ra_\tF\dx\\
\leq&-\frac 12\int_\Omega\int_0^\tau\|\ps W_r^n \|_\tF^2\ds\dx+\frac 12\int_\Omega \int_0^\tau\|\widetilde{P}_{r-1}^n (s)-\mNk [U^n]\|_\tF^2\ds\dx.
\end{aligned}
\end{equation}
For the first term on the right-hand side of \eqref{eq:thm3.5-5}, using the Cauchy–Schwarz inequality in time, we obtain
\begin{equation}\label{eq:thm3.5-6}
\begin{split}
\int_\Omega\int_0^\tau\|\ps W_r^n \|_\tF^2\ds\dx=&\sum_{i=1}^{m_1}\sum_{j=1}^{m_2}\int_\Omega\int_0^\tau(\ps W_r^n)_{i,j}^2\ds\dx\\
=&\frac{1}{\tau}\sum_{i=1}^{m_1}\sum_{j=1}^{m_2}\int_\Omega\Big(\int_0^\tau1\ds\int_0^\tau(\ps W_r^n)_{i,j}^2\ds\Big)\dx\\
\geq&\frac{1}{\tau}\sum_{i=1}^{m_1}\sum_{j=1}^{m_2}\int_\Omega\Big(\int_0^\tau(\ps W_r^n)_{i,j}\ds\Big)^2\dx\\
=&\frac{1}{\tau}\int_\Omega\|U^{n+1}-U^n\|_\tF^2\dx.
\end{split}
\end{equation}
For the second term on the right-hand side of \eqref{eq:thm3.5-5}, using the triangle inequality, we have
\begin{equation}\label{eq:thm3.5-7-0}
\int_\Omega \int_0^\tau\|\widetilde{P}_{r-1}^n (s)-\mNk [U^n]\|_\tF^2\ds\dx
\leq2\int_\Omega\int_0^\tau\Big(\|P_{r-1}^n(s)-\mNk[U^n]\|_\tF^2+\|\widetilde{P}_{r-1}^n (s)-P_{r-1}^n (s)\|_\tF^2\Big)\ds\dx.
\end{equation}
Then, according to the definitions of $\widetilde{P}_{r-1}^n(s)$ and $P_{r-1}^n(s)$ and using $0\leq s\leq \tau$, we obtain
\begin{equation}\label{eq:thm3.5-7-1}
\|P_{r-1}^n(s)-\mNk[U^n]\|_\tF=\Big\|\sum_{k=1}^{r-1}C_{r-1,k}^n\Big(\frac s\tau\Big)^k\Big\|_\tF\leq \sum_{k=1}^{r-1}\Big\|C_{r-1,k}^n\Big\|_\tF\Big(\frac s\tau\Big)^k\leq\sum_{k=1}^{r-1}\|C_{r-1,k}^n\|_\tF,
\end{equation}
and
\begin{equation}\label{eq:thm3.5-7-2}
\begin{aligned}
&\|\widetilde{P}_{r-1}^n (s)-P_{r-1}^n (s)\|_\tF=\max\left\{\frac{\max\limits_{s\in[0,\tau]}\|P_{r-1}^n(s,\bx)\|_\tF-\kappa\sqrt{m_2}}{\max\limits_{s\in[0,\tau]}\|P_{r-1}^n(s,\bx)\|_\tF},0\right\}\|P_{r-1}^n(s,\bx)\|_\tF\\
\leq&\max\{\max\limits_{s\in[0,\tau]}\|P_{r-1}^n(s,\bx)\|_\tF-\kappa\sqrt{m_2},0\}
\leq\max\limits_{s\in[0,\tau]}\|P_{r-1}^n(s)-\mNk[U^n]\|_\tF\\
=&\max\limits_{s\in[0,\tau]}\Big\|\sum_{k=1}^{r-1}C_{r-1,k}^n\Big(\frac{s}{\tau}\Big)^k\Big\|_\tF\leq\sum_{k=1}^{r-1}\|C_{r-1,k}^n\|_\tF.
\end{aligned}
\end{equation}
Combining \eqref{eq:thm3.5-7-1} and \eqref{eq:thm3.5-7-2} with \eqref{eq:thm3.5-7-0} and using the Cauchy--Schwarz inequality and the definition of $\mathbf{C}_{r-1}^n$ in \eqref{matrix:C}, we obtain
\begin{equation}\label{eq:thm3.5-7}
\int_\Omega \int_0^\tau\|\widetilde{P}_{r-1}^n (s)-\mNk [U^n]\|_\tF^2\ds\dx\leq4\tau(r-1)\int_\Omega\sum_{k=1}^{r-1}\|C_{r-1,k}^n\|^2_\tF\dx
=4\tau(r-1)\int_\Omega\|\mathbf{C}_{r-1}^n\|_\tF^2\dx.
\end{equation}

Next, combining \eqref{eq:thm3.5-5}-\eqref{eq:thm3.5-7} leads to
\begin{equation*}
E(U^{n+1})-E(U^n)
\leq 2\tau(r-1)\|\mathbf{C}_{r-1}^n\|_{L^2(\Omega)}^2-\frac{1}{2\tau}\|U^{n+1}-U^n\|_{L^2(\Omega)}^2.
\end{equation*}
Therefore, we transform the proof of $E(U^{n+1})-E(U^n)\leq 0$ into proving
\begin{equation*}
2\tau\sqrt{r-1}\|\mathbf{C}_{r-1}^n\|_{L^2(\Omega)}\leq\|U^{n+1}-U^n\|_{L^2(\Omega)}.
\end{equation*}

To prove this inequality, it's sufficient to show
\begin{equation}\label{eq:thm3.5-9}
2\tau\sqrt{r-1}\|\mathbf{C}_{r-1}^n\|_{L^2(\Omega)}+\|U^{n+1}-W_1^n(\tau)\|_{L^2(\Omega)}\leq\|W_1^n(\tau)-U^n\|_{L^2(\Omega)},
\end{equation}
where we used the triangle inequality 
\begin{equation*}
\|U^{n+1}-U^n\|_{L^2(\Omega)}\geq \|W_1^n(\tau)-U^n\|_{L^2(\Omega)}-\|U^{n+1}-W_1^n(\tau)\|_{L^2(\Omega)}.
\end{equation*}
According to the explicit formula of the rescaled ETDRK$r$ scheme \eqref{eq:RETDRKr} and using Lemma \ref{lem:phi}, we have
\begin{equation}\label{eq:thm3.5-10}
\begin{aligned}
&\|U^{n+1}-W_1^n(\tau)\|_{L^2(\Omega)}\\
=&\Big\|\tau\phi_1(\tau\mL_\kappa)\left((\alpha_{r-1}^n-1)\mNk[U^n]\right)+\tau\sum_{k=1}^{r-1}k!\phi_{k+1}(\tau\mL_\kappa)\left(\alpha_{r-1}^nC_{r-1,k}^n\right)\Big\|_{L^2(\Omega)}\\
\leq&\tau\|\phi_1(\tau\mL_\kappa)\left((\alpha_{r-1}^n-1)\mNk[U^n]\right)\|_{L^2(\Omega)}+\tau\Big\|\sum_{k=1}^{r-1}k!\phi_{k+1}(\tau\mL_\kappa)\left(\alpha_{r-1}^nC_{r-1,k}^n\right)\Big\|_{L^2(\Omega)}.
\end{aligned}
\end{equation}
Using Lemma \ref{lem:phi} and the calculations in  \eqref{eq:thm3.5-7-2} for $s=0$, and the Cauchy--Schwarz inequality, we obtain
\begin{equation}\label{eq:thm3.5-11}
\begin{split}
\|\phi_1(\tau\mL_\kappa)\left((\alpha_{r-1}^n-1)\mNk[U^n]\right)\|_{L^2(\Omega)}\leq&\|(\alpha_{r-1}^n-1)P_{r-1}^n(0,\bx)\|_{L^2(\Omega)}\\
\leq&\max\limits_{s\in[0,\tau]}\|P_{r-1}^n(s,\bx)-\mNk[U^n]\|_{L^2(\Omega)}\\
=&\max\limits_{s\in[0,\tau]}\Big\|\sum_{k=1}^{r-1}C_{r-1,k}^n\Big(\frac{s}{\tau}\Big)^k\Big\|_{L^2(\Omega)}\\
\leq&\sqrt{r-1}\|\mathbf{C}_{r-1}^n\|_{L^2(\Omega)},
\end{split}
\end{equation}
and
\begin{equation}\label{eq:thm3.5-12}
\begin{aligned}
\Big\|\sum_{k=1}^{r-1}k!\phi_{k+1}(\tau\mL_{\kappa})\left(\alpha_{r-1}^nC_{r-1,k}^n\right)\Big\|_{L^2(\Omega)}
\leq&\sum_{k=1}^{r-1}\|k!\phi_{k+1}(\tau\mL_{\kappa})\left(\alpha_{r-1}^nC_{r-1,k}^n\right)\|_{L^2(\Omega)}\\
\leq&\sum_{k=1}^{r-1}\frac{1}{k+1}\|\alpha_{r-1}^nC_{r-1,k}^n\|_{L^2(\Omega)}\\
\leq&\sum_{k=1}^{r-1}\frac{1}{k+1}\|C_{r-1,k}^n\|_{L^2(\Omega)}\\
\leq&\frac{1}{2}\sqrt{r-1}\Big(\sum_{k=1}^{r-1}\|C_{r-1,k}^n\|_{L^2(\Omega)}^2\Big)^\frac{1}{2}\\
=&\frac 12\sqrt{r-1}\|\mathbf{C}_{r-1}^n\|_{L^2(\Omega)}.
\end{aligned}
\end{equation}
Combining \eqref{eq:thm3.5-10}--\eqref{eq:thm3.5-12}, we obtain
\begin{equation}\label{eq:thm3.5-13}
\|U^{n+1}-W_1^n(\tau)\|_{L^2(\Omega)}
\leq \frac{3}{2}\tau\sqrt{r-1}\|\mathbf{C}_{r-1}^n\|_{L^2(\Omega)}.
\end{equation}
Thus, from \eqref{eq:thm3.5-13} and \eqref{eq:thm3.5-9}, we can infer that if the following inequality
\begin{equation}\label{eq:thm3.5-14}
\frac 72\tau\sqrt{r-1}\|\mathbf{C}_{r-1}^n\|_{L^2(\Omega)}\leq\|W_1^n(\tau)-U^n\|_{L^2(\Omega)}.
\end{equation}
holds, then $E(U^{n+1})\leq E(U^n)$. The right-hand side of the inequality \eqref{eq:thm3.5-14} is independent of $r$. When $r=1$, the left-hand side of the inequality \eqref{eq:thm3.5-14} is equal to $0$. Thus, the inequality holds for any $\tau>0$. 

Suppose that there exists a certain constant $\tau_{\max,r}>0$, such that the inequality \eqref{eq:thm3.5-14} holds for $\tau<\tau_{\max,r}$. Then, we will prove that there exists a certain constant $\tau_{\max,r+1}>0$, such that
\begin{equation}\label{eq:thm3.5-14_2}
\frac 72\tau\sqrt{r}\|\mathbf{C}_{r}^n\|_{L^2(\Omega)}\leq\|W_1^n(\tau)-U^n\|_{L^2(\Omega)}.
\end{equation}
From Lemma \ref{lem:martix_C}, we know that
\begin{equation*}
\begin{split}
\frac 72\tau\sqrt{r}\|\mathbf{C}_r^n\|_{L^2(\Omega)}\leq& \frac{7\kappa\tau\sqrt{r}}{\sigma_{\min}(V_r)}\Big(\sum_{k=1}^r\|W_r^n(a_{r,k}\tau)-U^n\|_{L^2(\Omega)}^2\Big)^{\frac 12}.
\end{split}
\end{equation*}
According to the explicit formula of rescaled ETDRK$r$ \eqref{eq:RETDRKr}, for $k=1,2,\ldots,r$, we have
\begin{equation}\label{eq:thm3.5-15}
\begin{split}
\|W_r^n(a_{r,k}\tau)-U^n\|_{L^2(\Omega)}
\leq&\|W_1^n(a_{r,k}\tau)-U^n\|_{L^2(\Omega)}+\|\tau\phi_1(a_{r,k}\tau\mL_\kappa)\left((\alpha_{r-1}^n-1)\mNk[U^n]\right)\|_{L^2(\Omega)}\\
&+\Big\|\tau\sum_{j=1}^{r-1}j!(a_{r,k})^{j+1}\phi_{j+1}(a_{r,k}\tau\mL_{\kappa})\left(\alpha_{r-1}^nC_{r-1,j}^n\right)\Big\|_{L^2(\Omega)}.
\end{split}
\end{equation}
Using the explicit formula of ETDRK1 method \eqref{eq:ETDRK1} and Lemma \ref{lem:phi}, we have
\begin{equation}\label{eq:thm3.5-16}
\begin{aligned}
\|W_1^n(a_{r,k}\tau)-U^n\|_{L^2(\Omega)}=&\|e^{a_{r,k}\tau\mL_\kappa}U^n+(e^{a_{r,k}\tau\mL_\kappa}-\mI)\mL_\kappa^{-1}\mNk[U^n]-U^n\|_{L^2(\Omega)}\\
=&\|a_{r,k}\tau\phi_1(a_{r,k}\tau\mL_\kappa)(\mL_\kappa U^n+\mNk[U^n])\|_{L^2(\Omega)}\\
\leq&\|\tau\phi_1(\tau\mL_\kappa)(\mL_\kappa U^n+\mNk[U^n])\|_{L^2(\Omega)}\\
=&\|W_1^n(\tau)-U^n\|_{L^2(\Omega)},
\end{aligned}
\end{equation}
for $k=1,\ldots,r$. Similar to \eqref{eq:thm3.5-11} and \eqref{eq:thm3.5-12}, we obtain
\begin{equation}\label{eq:thm3.5-17}
\|\tau\phi_1(a_{r,k}\tau\mL_\kappa)\left((\alpha_{r-1}^n-1)\mNk[U^n]\right)\|_{L^2(\Omega)}\leq\tau\sqrt{r-1}\|\mathbf{C}_{r-1}^n\|_{L^2(\Omega)},
\end{equation}
and
\begin{equation}\label{eq:thm3.5-18}
\Big\|\tau\sum_{j=1}^{r-1}j!(a_{r,k})^{j+1}\phi_{j+1}(a_{r,k}\tau\mL_{\kappa})\left(\alpha_{r-1}^nC_{r-1,j}^n\right)\Big\|_{L^2(\Omega)}\leq\frac 12\tau\sqrt{r-1}\|\mathbf{C}_{r-1}^n\|_{L^2(\Omega)},
\end{equation}
for $k=1,\ldots,r$. Combining \eqref{eq:thm3.5-15}--\eqref{eq:thm3.5-18} yields
\begin{equation*}
\Big(\sum_{k=1}^r\|W_r^n(a_{r,k}\tau)-U^n\|_{L^2(\Omega)}^2\Big)^{\frac 12}
\leq \sqrt{r}\Big(\|W_1^n(\tau)-U^n\|_{L^2(\Omega)}+\frac{3}{2}\tau\sqrt{r-1}\left\|\mathbf{C}_{r-1}^n\right\|_{L^2(\Omega)}\Big).
\end{equation*}
Using Lemma \ref{lem:martix_C} and combining the inductive hypothesis \eqref{eq:thm3.5-14}, we obtain
\begin{equation*}
\frac 72\tau\sqrt{r}\|\mathbf{C}_r^n\|_{L^2(\Omega)}\leq\frac{10\kappa\tau r}{\sigma_{\min}(V_r)}\|W_1^n(\tau)-U^n\|_{L^2(\Omega)}.
\end{equation*}
In order to make \eqref{eq:thm3.5-14_2} hold, it is sufficient to satisfy the following inequality:
\begin{equation*}
\frac{10\kappa\tau r}{\sigma_{\min}(V_r)}\leq 1.
\end{equation*}
Thus, the inequality \eqref{eq:thm3.5-14_2} holds for any
\begin{equation*}
\tau\leq\tau_{\max,r+1}:=\min\Big\{\tau_{\max,r},\frac{\sigma_{\min}(V_r)}{10\kappa r}\Big\}.
\end{equation*}
Finally, by induction, we finish the proof.
\end{proof}

\subsection{Convergence analysis}
By Theorem \ref{thm:mbp}, we have the exact solution of \eqref{eq:MAC} $\|U(t_n+s)\|_\tF\leq \sqrt{m_2}$, {for any $(s,\bx)\in[0,\tau]\times\Omega$.} Then, using Lemma \ref{Nkappa}, we can obtain that
\begin{equation}\label{eq:sec3.3-0}
\|\mNk[U(t_n+s,\bx)]\|_\tF\leq\kappa\sqrt{m_2}, \quad \forall~{(s,\bx)\in[0,\tau]\times\Omega,}
\end{equation}
for $\kappa\geq \max\{\frac 32 m_2-1,2\}$.

According to the definition of scaling factor $\alpha_r^n(\bx)$ and using \eqref{eq:sec3.3-0}, we can obtain the following results about $P_r^n(s,\bx)$ and $\widetilde{P}_r^n(s,\bx)$.

\begin{lemma}\label{lem:P-P}
Let $U(t,\bx)$ be the exact solution of the generalized matrix-valued Allen--Cahn
equation satisfying $\|U(t)\|_{\mathcal X}\le\sqrt{m_2}$ for
$t\in[t_n,t_{n+1}]$. If $\max\limits_{s\in[0,\tau]}\|P_r^n(s,\bx)\|_\tF\leq\kappa\sqrt{m_2}$, then $\alpha_r^n(\bx)=1$, i.e., $\widetilde{P}_r^n(s,\bx)=P_r^n(s,\bx)$. If $\max\limits_{s\in[0,\tau]}\|P_r^n(s,\bx)\|_\tF>\kappa\sqrt{m_2}$, $\kappa\geq 3m_2+1$ and $\{s_k=k\frac{\tau}{r}\}_{k=0}^r$ is the set of uniform nodes in $[0,\tau]$, then
for any $s\in[0,\tau]$ we have
\begin{equation}\label{eq:lem3.6-0}
\|P_r^n(s,\bx)-\widetilde{P}_r^n(s,\bx)\|_\mX\leq 2\kappa r^r\sum_{k=0}^r\|W_r^n(s_k,\bx)-U(t_n+s_k,\bx)\|_\mX+C_r\tau^{r+1},
\end{equation}
where $C_r$  is a constant depending on $\kappa$, $m_2$, $r$, $t_n$ and the $C^{r+1}([0,T];\mX)$-norm of $U$, but is independent of $\tau$.
\end{lemma}
\begin{proof}
According to the definition of scaling factor $\alpha_{r}^n(\bx)=\min\left\{\frac{\kappa\sqrt{m_2}}{\max\limits_{s\in[0,\tau]}\|P_r^n(s,\bx)\|_\tF},1\right\}$ for $\bx\in\Omega$ and polynomial $\widetilde{P}_r^n(s,\bx)=\alpha_r^n(\bx)P_r^n(s,\bx)$, if $\max\limits_{s\in[0,\tau]}\|P_r^n(s,\bx)\|_\tF\leq\kappa\sqrt{m_2}$, then $\alpha_r^n(\bx)=1$, i.e., $\widetilde{P}_r^n(s,\bx)=P_r^n(s,\bx)$. Moreover, if $\max\limits_{s\in[0,\tau]}\|P_r^n(s,\bx)\|_\tF>\kappa\sqrt{m_2}$, we can obtain
\begin{equation}\label{eq:lem3.6-1}
\begin{aligned}
\|P_r^n(s,\bx)-\widetilde{P}_r^n(s,\bx)\|_\tF&=\frac{\max\limits_{s\in[0,\tau]}\|P_r^n(s,\bx)\|_\tF-\kappa\sqrt{m_2}}{\max\limits_{s\in[0,\tau]}\|P_r^n(s,\bx)\|_\tF}\|P_r^n(s,\bx)\|_\tF\\
&\leq\max\limits_{s\in[0,\tau]}\|P_r^n(s,\bx)\|_\tF-\kappa\sqrt{m_2}.
\end{aligned}
\end{equation}
Then, combining \eqref{eq:sec3.3-0} and \eqref{eq:lem3.6-1} yields
\begin{gather}\label{eq:lem3.6-2}
\|P_r^n(s)-\widetilde{P}_r^n(s)\|_\mX\leq\max\limits_{s\in[0,\tau]}\|P_r^n(s)-\mNk[U(t_n+s)]\|_\mX\\
\leq\max\limits_{s\in[0,\tau]}\Big(\Big\|P_r^n(s)-\prod_{r}\mNk[U(t_n+s)]\Big\|_\mX+\Big\|\prod_{r}\mNk[U(t_n+s)]-\mNk[U(t_n+s)]\Big\|_\mX\Big),\no  
\end{gather}
where $\prod_{r}$ represents the Lagrange interpolation operator. For the Lagrange basis function $\ell_{r,k}(s)$, it holds that
\begin{equation}\label{eq:lem3.6-3}
|\ell_{r,k}(s)|=\Bigg|\prod_{\substack{i=0\\i\neq k}}^r\frac{s-s_i}{s_k-s_i}\Bigg|\leq {\Big(\frac{\tau}{\tau/r}\Big)^r = r^r}.
\end{equation}
Using Lemma \ref{Nkappa} and inequality \eqref{eq:lem3.6-3}, we obtain
\begin{equation}\label{eq:lem3.6-4}
\begin{aligned}
\Big\|P_r^n(s,\bx)-\prod_{r}\mNk[U(t_n+s)]\Big\|_\mX
\leq&\sum_{k=0}^r|\ell_{r,k}(s)|\|\mNk[W_r^n(s_k)]-\mNk[U(t_n+s_k)]\|_\mX\\
\leq& 2\kappa r^r\sum_{k=0}^r\|W_r^n(s_k)-U(t_n+s_k)\|_\mX.
\end{aligned}
\end{equation}

By Lemma \ref{Nkappa} and Taylor's formula, we have
\begin{equation}\label{eq:lem3.6-5}
\begin{split}
&\Big\|\prod_{r}\mNk[U(t_n+s)]-\mNk[U(t_n+s)]\Big\|_\mX\\
=&\Big\|\sum_{k=0}^r\ell_{r,k}(s)\mNk[U(t_n+s_k)]-\mNk[U(t_n+s)]\Big\|_\mX\\
\leq&\hat{C}\sum_{k=0}^{r+1}\Big(\sum_{i=1}^{m_1}\sum_{j=1}^{m_2}\Big(\max_{x\in\overline{\Omega}}|U_{i,j}^{(r+1)}(t_n+\xi_{ij}^k)|\Big)^2\Big)^{\frac 12}\tau^{r+1}+\tilde{C}(\tau^{r+2}+\cdots+\tau^{3(r+1)})\\
\leq& C\tau^{r+1},
\end{split}
\end{equation}
where $U=(U_{ij})_{m_1\times m_2}$ and $\xi_{ij}^k\in[0,\tau]$ for $k=0,\ldots,r+1$, $U_{ij}^{(r)}(t)$ denotes the $r$th-derivative of $U_{ij}(t)$ for $i=1,\ldots,m_1$ and $j=1,\ldots,m_2$, the constant $\hat{C}$ depends on $\kappa$, $m_2$ and $r$, $\tilde{C}$ depends on $m_2$ and $C^{r+1}([0,T];\mX)$-norm of $U$, $C$ depends on $\kappa$, $m_1$, $m_2$, $r$, $t_n$, and $C^{r+1}([0,T];\mX)$-norm of $U$.

Finally, substituting \eqref{eq:lem3.6-4} and \eqref{eq:lem3.6-5} into \eqref{eq:lem3.6-2} yields \eqref{eq:lem3.6-0}.
\end{proof}

\begin{remark}
The bound in \eqref{eq:lem3.6-3} is a simple estimate for uniform interpolation nodes and is not sharp with respect to $r$. Using Chebyshev--Lobatto nodes could improve the corresponding interpolation bound to $\mathcal O(\log r)$ \cite{trefethen2019approximation,brutman1997lebesgue}. In the present convergence analysis, $r$ is fixed while $\tau$ tends to zero, so the choice of nodes affects only the $r$-dependent constants and does not change the designed temporal order.
\end{remark}

\begin{theorem}[Convergence analysis]\label{thm:conv}
Assume that $\kappa\geq3m_2+1$ and $\|U_0\|_\tF\leq \sqrt{m_2}$. Let $U(t)\in C^{r+1}([0,T];\mathcal X)$ be the exact solution of \eqref{eq:MAC} on $[0,T]$, and let $\{U^n\}_{n\ge0}$ be the numerical solution generated by the rescaled ETDRK$r$ scheme with uniform time step $\tau = \frac{T}{N}$. Then, we have
\begin{equation*}
\|U^n-U(t_n)\|_\mX\leq K_r\tau^re^{\kappa T_nJ_r},
\end{equation*}
for any $\tau>0$, where the constants $K_r$,$J_r>0$ are independent of $\tau$.
\end{theorem}
\begin{proof}
Let $e_r^n:=U_r^n-U(t_n)$ denote the error of the rescaled ETDRK$r$ scheme. We first prove the case of the ETDRK$1$ method. Subtracting \eqref{eq:MAC_ka} from \eqref{eq:MAC-W} yields
\begin{equation}\label{eq:thm3.7-1}
\begin{aligned}
e_1^{n+1}=e^{\tau\mathcal{L}_\kappa}e_1^n+\int_0^{\tau}e^{(\tau-s)\mathcal{L}_\kappa}\left(\mNk[U^n]-\mNk[U(t_n)]+R_1(s)\right)\ds,
\end{aligned}
\end{equation}
where $R_1(s)=\mNk[U(t_n)]-\mNk[U(t_n+s)],~s\in [0,\tau]$.
Then, taking the supremum norm $\|\cdot\|_{\mathcal{X}}$ to $R_1(s)$ and using Lemma \ref{Nkappa}, we derive
\begin{equation}\label{eq:thm3.7-2}
\begin{aligned}
\|R_1(s)\|_{\mathcal{X}}&=\max _{\bx  \in \overline{\Omega}}\|\mNk[U(t_n)]-\mNk[U(t_n+s)]\|_\tF\\
&\leq2\kappa\max _{\bx \in \overline{\Omega}}\|U(t_n)-U(t_n+s)\|_\tF\\
&\leq2\kappa\left(\sum_{i=1}^{m_1}\sum_{j=1}^{m_2} \left(\max_{\bx \in \overline{\Omega}}\left|U_{ij}^{'}(t_n+\xi_{ij})\right|\right)^2\right)^{\frac 12}s\\
&\leq C\sqrt{m_1m_2}\kappa\tau,
\end{aligned}
\end{equation}
where $U=(U_{ij})_{m_1\times m_2}$, $\xi_{ij}\in[0,\tau]$, and the constant $C$ is the $C^1([0,T];\mathcal{X})$-norm of $U$. Similarly, according to Theorem \ref{thm:MBP} and Lemma \ref{Nkappa}, we have $\|U^n\|_\tF\leq\sqrt{m_2}$ and 
\begin{equation}\label{eq:thm3.7-3}
\begin{aligned}
\|\mNk[U^n]-\mNk[U(t_n)]\|_{\mathcal{X}}\leq2\kappa\|U^n-U(t_n)\|_{\mathcal{X}}=2\kappa\|e_1^n\|_{\mathcal{X}}.
\end{aligned}
\end{equation}
Taking the supremum norm $\|\cdot\|_{\mathcal{X}}$ to \eqref{eq:thm3.7-1} and using \eqref{eq:thm3.7-2}--\eqref{eq:thm3.7-3} and Lemma \ref{lambda-Delta}, we can derive
\begin{equation*}
\begin{aligned}
\|e_1^{n+1}\|_{\mathcal{X}}\leq& e^{-\kappa\tau}\|e_1^n\|_{\mathcal{X}}+\int_0^{\tau}e^{-\kappa(\tau-s)}(\|\mNk[U^n]-\mNk[U(t_n)]\|_{\mathcal{X}}+\|R_1(s)\|_{\mathcal{X}})\ds\\
\leq&e^{-\kappa\tau}\|e_1^n\|_{\mathcal{X}}+\left(2\kappa\|e_1^n\|_{\mathcal{X}}+C\sqrt{m_1m_2}\kappa\tau\right)\int_0^{\tau}e^{-\kappa(\tau-s)}\ds\\
=&e^{-\kappa\tau}\|e_1^n\|_{\mathcal{X}}+\left(2\kappa\|e_1^n\|_{\mathcal{X}}+C\sqrt{m_1m_2}\kappa\tau\right)\frac{1-e^{-\kappa\tau}}{\kappa}\\
\leq&(1+\kappa\tau)\|e_1^n\|_{\mathcal{X}}+C\sqrt{m_1m_2}\kappa\tau^2,
\end{aligned}
\end{equation*}
where we use $1-e^{-a}\leq a$ for any $a>0$.

Supposed that there exist positive constants $M_r$, $C_{r,1}$, \ldots, $C_{r,r}$ independent of $\tau$, such that 
\begin{equation}\label{eq:thm3.7-4}
\begin{aligned}
\|e_r^{n+1}\|_\mX\leq&(1+C_{r,1}\kappa\tau+\ldots+C_{r,r}\kappa^r\tau^r)\|e_r^n\|_\mX+M_r\tau^{r+1}.
\end{aligned}
\end{equation}
Then, we will prove that there exist positive constants $M_{r+1}$, $C_{r+1,1}$, \ldots, $C_{r+1,r+1}$ independent of $\tau$, such that 
\begin{equation}\label{eq:error_r+1}
\begin{aligned}
\|e_{r+1}^{n+1}\|_\mX\leq\Big(1+C_{r+1,1}\kappa\tau+\ldots+C_{r+1,r+1}\kappa^{r+1}\tau^{r+1}\Big)\|e_{r+1}^n\|_\mX+M_{r+1}\tau^{r+2}.
\end{aligned}
\end{equation}
Subtracting \eqref{eq:mac-solution} from \eqref{eq:RETDRKr} yields
\begin{equation}\label{eq:thm3.7-5}
e_{r+1}^{n+1}=e^{\tau\mL_{\kappa}}e_{r+1}^n+\int_0^{\tau}e^{(\tau-s)\mL_{\kappa}}\left(\widetilde{P}_{r}^n (s)-\mNk[U(t_n+s)]\right)\ds.
\end{equation}
Taking the supremum norm $\|\cdot\|_\mX$ to \eqref{eq:thm3.7-5} and using Lemma \ref{lambda-Delta}, we derive
\begin{equation}\label{eq:thm3.7-6}
\begin{aligned}
\|e_{r+1}^{n+1}\|_\mX=&\Big\|e^{\tau\mL_{\kappa}}e_{r+1}^n+\int_0^{\tau}e^{(\tau-s)\mL_{\kappa}}\Big(\widetilde{P}_{r}^n (s)-\mNk[U(t_n+s)]\Big)\ds\Big\|_\mX\\
\leq& e^{-\kappa\tau}\left\|e_{r+1}^n\right\|_\mX+\int_0^{\tau}e^{-\kappa(\tau-s)}\|\widetilde{P}_{r}^n (s)-\mNk[U(t_n+s)]\|_\mX \ds\\
\leq& e^{-\kappa\tau}\left\|e_{r+1}^n\right\|_\mX+\int_0^{\tau}e^{-\kappa(\tau-s)}\Big(\Big\|P_{r}^n (s)-\prod_{r}\mNk[U(t_n+s)]\Big\|_\mX +R_{r+1}^{'}(s)\Big)\ds,
\end{aligned}
\end{equation}
where $\prod_{r}$ represents the Lagrange interpolation operator corresponding to the uniform nodes $\{s_k\}_{k=0}^r$ and $R_{r+1}^{'}(s)$ is the truncation error given by 
\begin{equation*}
R_{r+1}^{'}(s)=\Big\|\widetilde{P}_{r}^n (s)-P_{r}^n(s)\Big\|_\mX+\Big\|\prod_{r}\mNk[U(t_n+s)]-\mNk[U(t_n+s)]\Big\|_\mX.
\end{equation*}
By Lemma \ref{lem:P-P} and \eqref{eq:lem3.6-5}, we have
\begin{equation}\label{eq:thm3.7-7}
R_{r+1}^{'}(s)\leq2\kappa r^r\sum_{k=0}^r\|W_r^n(s_k)-U(t_n+s_k)\|_\mX+\widetilde{C_r}\tau^{r+1},
\end{equation}
where $\widetilde{C_r}$ depends on $\kappa$, $m_2$, $r$, $t_n$, and $C^{r+1}([0,T];\mX)$-norm of $U$, but independent of $\tau$. Since we consider the $(r+1)$th method, the error at $t_n$ is $e_{r+1}^n$. We can replace $\tau$ with $s_k$ in \eqref{eq:thm3.7-4} and obtain that
\begin{equation}\label{eq:thm3.7-8}
\|W_r^n(s_k)-U(t_n+s_k)\|_\mX\leq\Big(1+\sum_{i=1}^rC_{r,i}\kappa^is_k^i\Big)\|e_{r+1}^n\|_\mX+M_rs_k^{r+1},
\end{equation}
for $k=0,\ldots,r$. Substituting \eqref{eq:lem3.6-4}, \eqref{eq:thm3.7-7} and \eqref{eq:thm3.7-8} into \eqref{eq:thm3.7-6} and using inequalities $1-e^{-a}\leq a$ for any $a\geq 0$, we have
\begin{equation*}
\begin{aligned}
\|e_{r+1}^{n+1}\|_\mX\leq& e^{-\kappa\tau}\left\|e_{r+1}^n\right\|_\mX+\frac{1-e^{-\kappa\tau}}{\kappa}\Big(4\kappa r^rM_r\sum_{k=0}^r\Big(\frac{k}{r}\Big)^{r+1}+\widetilde{C_r}\Big)\tau^{r+1}\\
&+\frac{1-e^{-\kappa\tau}}{\kappa}\Big(4\kappa r^r\sum_{k=0}^r\Big(1+\sum_{i=1}^rC_{r,i}\kappa^i\Big(k\frac{\tau}{r}\Big)^i\Big)\|e_{r+1}^n\|_\mX\Big)\\
\leq& e^{-\kappa\tau}\left\|e_{r+1}^n\right\|_\mX+\Big(4\kappa r^{r+1}M_r \int_{\frac{1}{r}}^{1+\frac{1}{r}}x^{r+1}\mathrm{d}x+\widetilde{C_r}\Big)\tau^{r+2}\\
&+\Big(4r^r(r+1)\kappa\tau+\sum_{i=1}^r 4r^{r+1}C_{r,i}(\kappa\tau)^{i+1}\int_{\frac{1}{r}}^{1+\frac{1}{r}}x^{i}\mathrm{d}x\Big)\left\|e_{r+1}^n\right\|_\mX\\
\leq&(1+C_{r+1,1}\kappa\tau+\cdots+C_{r+1,r+1}(\kappa\tau)^{r+1})\left\|e_{r+1}^n\right\|_\mX+M_{r+1}\tau^{r+2},
\end{aligned}
\end{equation*}
where $$M_{r+1}=\frac{4\kappa r^{r+1}M_r}{r+2}\Big(\Big(1+\frac{1}{r}\Big)^{r+2}-\Big(\frac{1}{r}\Big)^{r+2}\Big)+\widetilde{C_r},\qquad C_{r+1,1}=4r^r(r+1),$$
and 
$$C_{r+1,k}=\frac{4r^{r+1}C_{r,k-1}}{k}\Big(\Big(1+\frac{1}{r}\Big)^{k}-\Big(\frac{1}{r}\Big)^{k}\Big)$$
for $k=2,\ldots,r+1$. We proved that there exist positive constants $M_{r+1}$, $C_{r+1,1}$, \ldots,$C_{r+1,r+1}$ independent of $\tau$ such that \eqref{eq:error_r+1} holds. Therefore, our assumption is valid.

Finally, by induction for \eqref{eq:thm3.7-4}, we get
\begin{equation*}
\begin{aligned}
\|e_{r}^{n+1}\|_\mX\leq& \Big(1+\sum_{i=1}^{r}C_{r,i}\kappa^{i}\tau^i\Big)^n\|e_{r}^0\|_\mX+M_r\tau^{r+1}\sum_{k=0}^{n-1}\Big(1+\sum_{i=1}^{r}C_{r,i}\kappa^{i}\tau^i\Big)^k\\
=&M_r\tau^{r+1}\frac{\Big(1+\sum_{i=1}^{r}C_{r,i}\kappa^{i}\tau^i\Big)^n-1}{\sum_{i=1}^{r}C_{r,i}\kappa^{i}\tau^i}\leq K_r\tau^re^{T_nJ_r\kappa},
\end{aligned}
\end{equation*}
where $K_r=\frac{M_r}{\kappa C_{r,1}}$, $J_r=\max\limits_{1\leq k\leq r}k!C_{r,k}$ and $T_n=n\tau$.
\end{proof}

\section{Numerical Experiments}\label{sec4}
In this section, we present several numerical examples to verify the convergence, MBP preservation, and energy dissipation of the arbitrarily high-order rescaled ETDRK$r$ schemes for the generalized matrix-valued Allen--Cahn equation. The tests include the vector-valued case $(m_1,m_2)=(2,1)$ and two matrix-valued cases $(m_1,m_2)=(2,2)$ and $(3,2)$. For the vector-valued cases, we visualize the evolution of the vector field. For the matrix-valued cases, we simulate the evolution of the matrix field and its interface, defined as the zero-level set of its determinant. We also include a fully three-dimensional matrix-valued simulation to examine the behavior of the proposed schemes on a spatially three-dimensional domain. The products of matrix exponents with vectors are implemented by the fast Fourier transform.

We recall that our theoretical analysis of the rescaled ETDRK$r$ schemes establishes discrete energy dissipation under a time-step restriction $\tau_{\max,r}=\frac{1}{10\kappa}\min_{1\le j\le r-1}\frac{\sigma_{\min}(V_j)}{j}$ for $r\geq 3$ (see Theorem \ref{thm:energy}). In the numerical experiments, we do not explicitly enforce this constraint and also report results for relatively large time steps. In all tests, it can be observed that the discrete energy always decays monotonically, indicating that the theoretical time-step bound is a sufficient but not necessary condition and is unlikely to be optimal.
\subsection{Vector-valued case}
We first consider the vector-valued case $(m_1,m_2)=(2,1)$ of the generalized matrix-valued Allen--Cahn equation \eqref{eq:MAC} on $\Omega=[-\frac 12,\frac 12]^2$ with periodic boundary conditions and $\varepsilon=0.01$. The initial data are chosen as
\begin{equation}\label{eq:4.1}
U_0(x,y)=
\begin{bmatrix}
\cos(\alpha^2)\\
\sin(\alpha^2)
\end{bmatrix}
\end{equation}
where $\alpha$ denotes a uniformly distributed random variable. 

\subsubsection{Temporal convergence.}
We first test the temporal convergence of the rescaled ETDRK$r$ schemes for $r=3,4,5$ with $\kappa=5$. We take a sequence of time steps $\tau=0.1\times 2^{-k}$ with $k=0,1,\ldots,4$ to obtain the corresponding numerical solution at $t=1$. Furthermore, since the exact solution of \eqref{eq:MAC} is unavailable, we use a reference solution obtained with a very small time step $\tau=0.1\times 2^{-10}$. The $L^2$- and $L^{\infty}$-errors between these numerical solutions and the reference solution are shown in Table \ref{tab:1}. We observe that the convergence rates of the rescaled ETDRK$3$, ETDRK$4$, and ETDRK$5$ schemes are about $3$, $4$, and $5$, respectively, in agreement with Theorem \ref{thm:conv}.

\begin{table}[!htpb]
\centering
\normalsize
\caption{$L^2$- and $L^\infty$-errors and convergence rates of the rescaled ETDRK$r$ schemes ($r=3,4,5$) for the vector-valued Allen--Cahn equation \eqref{eq:MAC} at $t=1$.}
\label{tab:1}
\vspace{10pt}
\begin{tabular}{cccccc}
\hline
$r$&$\tau$ & $L^2$ error & Rate&$L^{\infty}$ error & Rate \\
\hline
\multirow{5}{*}{$r=3$}&
$0.1\times 2^{0}$&$2.43\times 10^{-2}$&-&$8.61\times 10^{-3}$&-\\
&$0.1\times 2^{-1}$&$3.76\times 10^{-3}$&2.68&$1.34\times 10^{-3}$&2.68\\
&$0.1\times 2^{-2}$&$5.26\times 10^{-4}$&2.84&$1.88\times 10^{-4}$&2.83\\
&$0.1\times 2^{-3}$&$6.95\times 10^{-5}$&2.92&$2.50\times 10^{-5}$&2.92\\
&$0.1\times 2^{-4}$&$8.94\times 10^{-6}$&2.96&$3.21\times 10^{-6}$&2.96\\
\hline
\multirow{5}{*}{$r=4$}&
$0.1\times 2^{0}$&$2.64\times 10^{-3}$&-&$9.64\times 10^{-4}$&-\\
&$0.1\times 2^{-1}$&$2.08\times 10^{-4}$&3.66&$7.69\times 10^{-5}$&3.65\\
&$0.1\times 2^{-2}$&$1.47\times 10^{-5}$&3.82&$5.45\times 10^{-6}$&3.82\\
&$0.1\times 2^{-3}$&$9.76\times 10^{-7}$&3.91&$3.62\times 10^{-7}$&3.91\\
&$0.1\times 2^{-4}$&$6.29\times 10^{-8}$&3.96&$2.34\times 10^{-8}$&3.95\\
\hline
\multirow{5}{*}{$r=5$}&
$0.1\times 2^{0}$&$2.40\times 10^{-4}$&-&$9.08\times 10^{-5}$&-\\
&$0.1\times 2^{-1}$&$9.58\times 10^{-6}$&4.64&$3.66\times 10^{-6}$&4.63\\
&$0.1\times 2^{-2}$&$3.39\times 10^{-7}$&4.82&$1.30\times 10^{-7}$&4.81\\
&$0.1\times 2^{-3}$&$1.13\times 10^{-8}$&4.90&$4.34\times 10^{-9}$&4.91\\
&$0.1\times 2^{-4}$&$3.86\times 10^{-10}$&4.88&$1.56\times 10^{-10}$&4.80\\
\hline
\end{tabular}
\end{table}

\subsubsection{Time-step robustness of the MBP and energy.} We examine the discrete MBP and energy dissipation for the rescaled ETDRK$r$ schemes with $r=3,4,5$. The terminal time is set to $T=100$ and the time steps are chosen as $\tau=0.25,0.5,1,2$. Figure \ref{fig:4.11} plots the evolution of the supremum norm and the discrete energy. For all combinations of $r$ and $\tau$, the value of $\|U^n\|_{\mathcal X}$ does not exceed $\sqrt{m_2}=1$, and the energy decreases monotonically.
\begin{figure}[!htpb]
\centering
\includegraphics[width=0.96\textwidth]{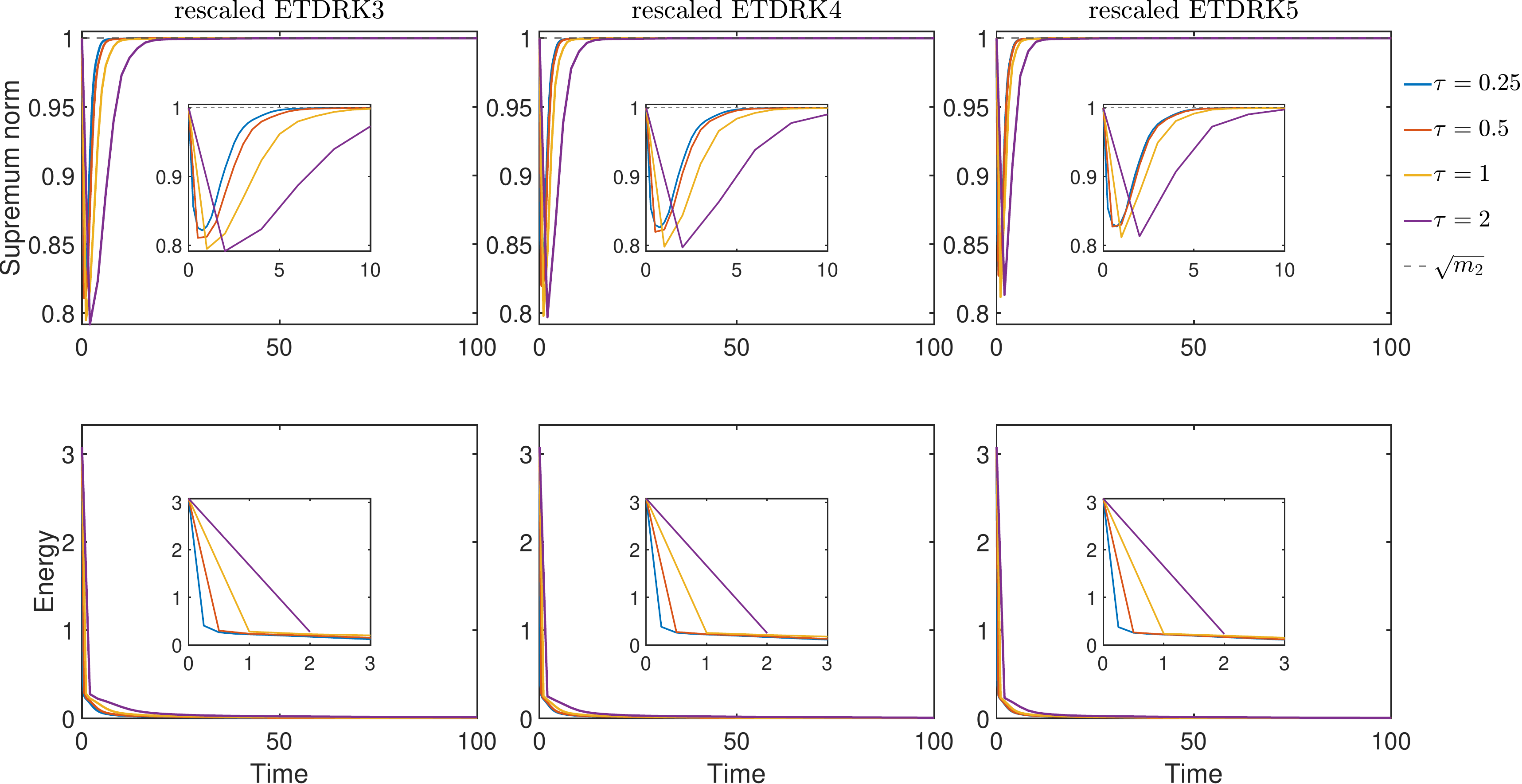}
\caption{Evolution of the supremum norm $\|\cdot\|_{\mathcal{X}}$ (upper row) and the discrete energy (lower row) for the rescaled ETDRK$r$ schemes ($r=3,4,5$) with initial condition \eqref{eq:4.1}. In each row, the three subfigures correspond to $r=3,4,5$ from left to right.}
\label{fig:4.11}
\end{figure}

\subsubsection{Sensitivity to the stabilizing parameter $\kappa$.} We test the influence of the stabilizing parameter $\kappa$ on the vector-valued problem. Since $\kappa$ is an auxiliary parameter introduced in the separation of linear and nonlinear terms, the continuous problem is independent of its value. The sufficient condition used in the energy and convergence analysis is $\kappa\geq3m_2+1$. We compare $\kappa=5,10,15$ using the rescaled ETDRK$5$ scheme with $\tau=1$ and $T=500$.
\begin{figure}[!htpb]
\centering
\includegraphics[width=0.70\textwidth]{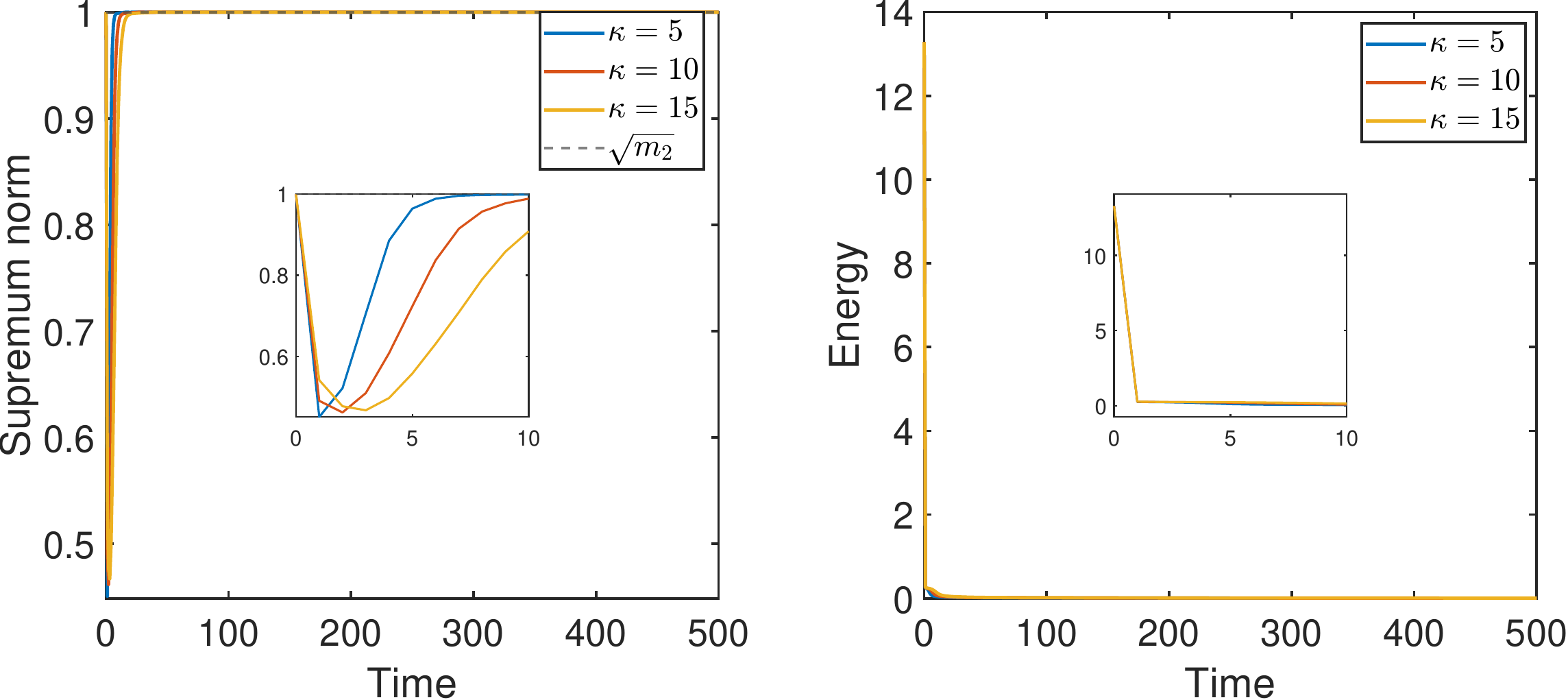}
\caption{Evolution of the supremum norm $\|\cdot\|_{\mathcal{X}}$ and the discrete energy for the rescaled ETDRK$5$ scheme with $\kappa=5,10,15$ up to $T=500$ with $\tau=1$.}
\label{fig:4.12}
\end{figure}

\begin{figure}[!htpb]
\centering
\subfigure[$\kappa=5$]{\includegraphics[width=0.98\textwidth]{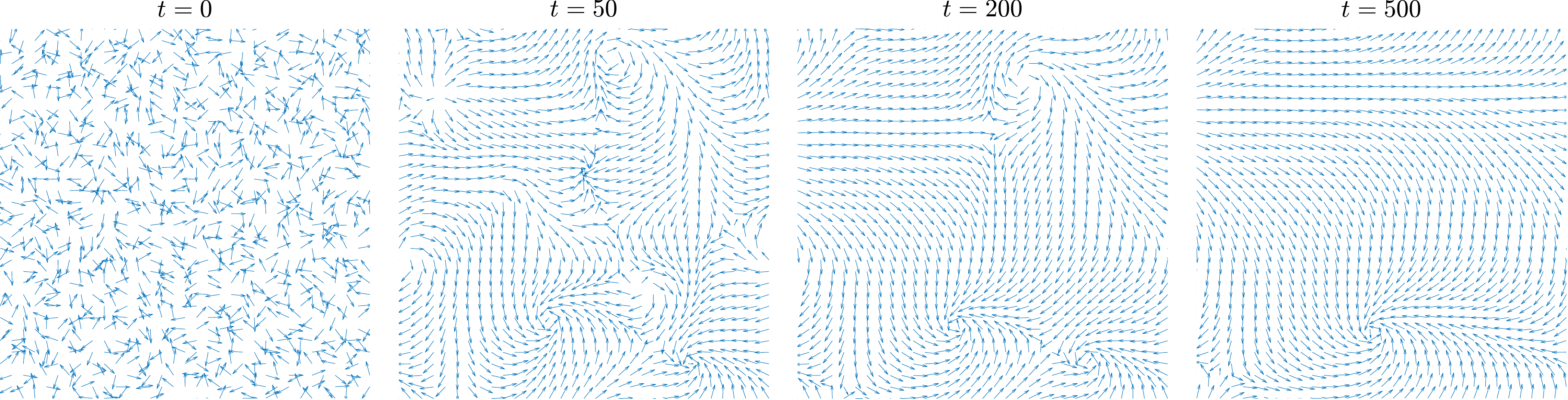}}\\[-1mm]
\subfigure[$\kappa=10$]{\includegraphics[width=0.98\textwidth]{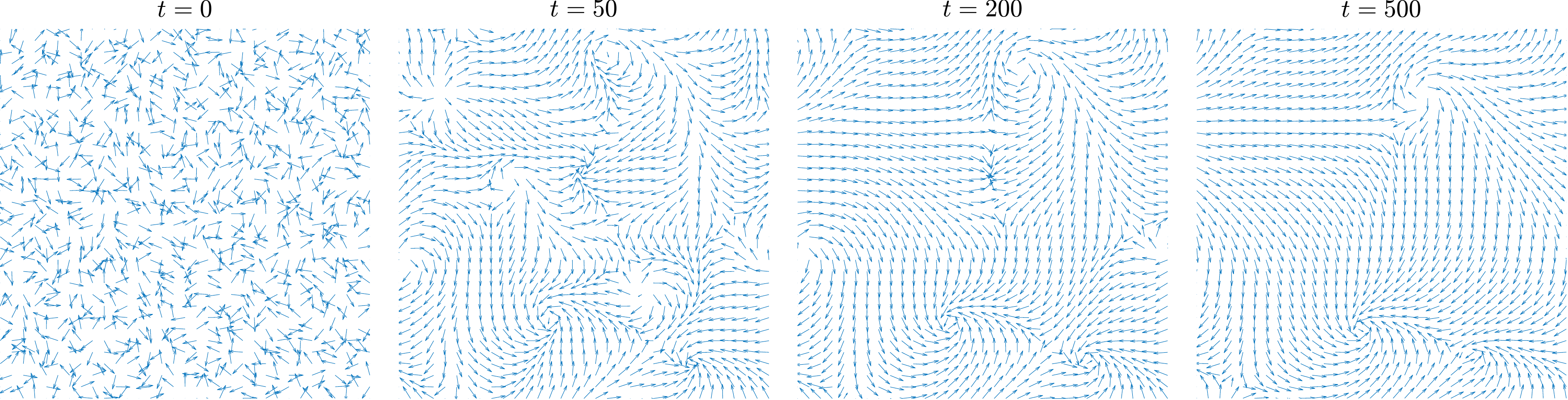}}\\[-1mm]
\subfigure[$\kappa=15$]{\includegraphics[width=0.98\textwidth]{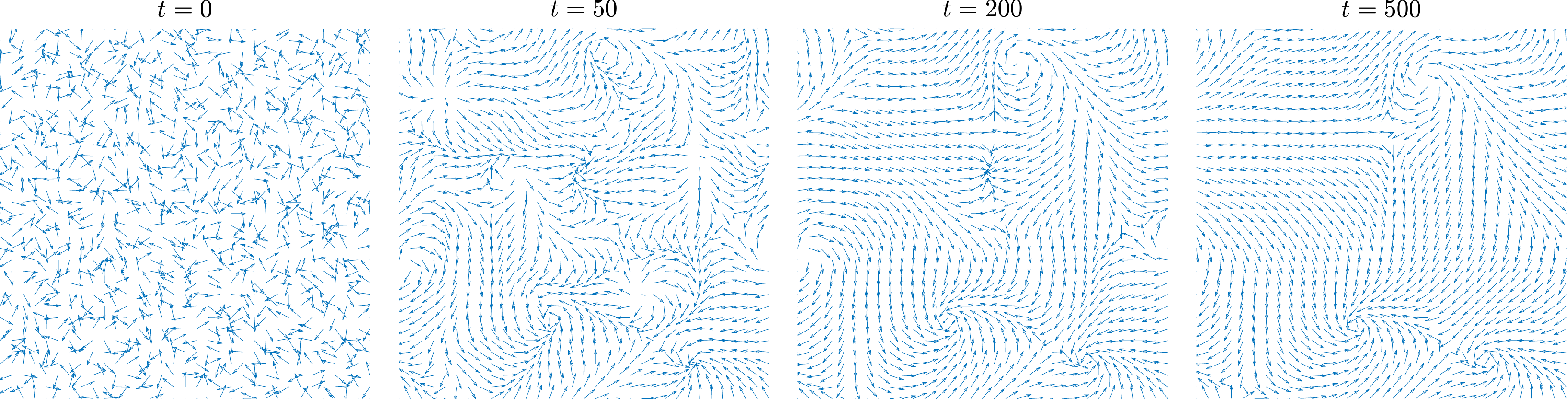}}
\caption{Evolution of the vector field for the rescaled ETDRK$5$ scheme with $\kappa=5,10,15$ at $t=0,50,200,500$. The initial field is given in \eqref{eq:4.1}.}
\label{fig:4.13}
\end{figure}

Figure \ref{fig:4.12} shows that the supremum norm remains below the theoretical bound $\sqrt{m_2}=1$ and that the discrete energy decreases monotonically for all tested values of $\kappa$. Thus, the MBP and energy dissipation properties are robust with respect to $\kappa$. However, as $\kappa$ increases, the recovery of the supremum norm toward $1$ is delayed. The snapshots in Figure \ref{fig:4.13} show the same qualitative evolution for all three values, but the speed of the vector field's evolution slows down as $\kappa$ increases. These observations suggest that, subject to the theoretical lower bound, a relatively small value of $\kappa$ is preferable in practice.

\subsubsection{Computational efficiency and structure-preserving comparison.}
We compare the proposed rescaled ETDRK$r$ schemes with the $r$th-order implicit-explicit backward differentiation method (IMEX-BDF$r$) and the implicit-explicit Runge--Kutta method (IMEX-RK$r$), for $r=3,4,5$. All three methods use the same FFT-based spatial discretization and the initial condition \eqref{eq:4.1}. The computations are performed on a $128\times128$ grid with $\varepsilon=0.01$, $\kappa=5$, and $T=100$, using $\tau=0.25,0.5,1,2$.

Let $M$ denote the total number of spatial grid points and let $q$ be the number of nonlinear or internal-stage evaluations in one time step. For an $m_1\times m_2$ matrix field, the forward and inverse FFTs cost $\mathcal O(qm_1m_2M\log M)$, while the pointwise matrix nonlinearity costs $\mathcal O(qm_1m_2^2M)$. The rescaling procedure requires only pointwise Frobenius-norm computations, maximum operations, and scalar multiplications, all of which are linear in the number of matrix entries and spatial grid points. Therefore, the rescaling procedure does not change the complexity with respect to $M$, $m_1$, and $m_2$. However, the constants in these methods differ. IMEX-BDF$r$ performs one nonlinear evaluation per step; the IMEX-RK scheme used in this paper has $4$, $6$, and $8$ stages for $r=3$, $4$, and $5$, respectively; while the rescaled ETDRK scheme requires the construction of recursive stages.
\begin{table}[!htpb]
\centering
\normalsize
\renewcommand{\arraystretch}{1.2}
\caption{Average CPU time per step (in seconds) for the rescaled ETDRK$r$, IMEX-BDF$r$, and IMEX-RK$r$ schemes ($r=3,4,5$) applied to the vector-valued Allen--Cahn equation with initial condition \eqref{eq:4.1}.}
\label{tab:vector_cost}
\vspace{10pt}
\begin{tabular}{ccccc}
\hline
$r$ & $\tau$ & rescaled ETDRK$r$ & IMEX-BDF$r$ & IMEX-RK$r$\\
\hline
\multirow{4}{*}{$r=3$}
& $0.25$ & $0.00440$ & $0.00047$ & $0.00186$\\
& $0.5$  & $0.00437$ & $0.00042$ & $0.00167$\\
& $1$    & $0.00466$ & $0.00062$ & $0.00188$\\
& $2$    & $0.00503$ & $0.00063$ & $0.00261$\\
\hline
\multirow{4}{*}{$r=4$}
& $0.25$ & $0.01030$ & $0.00048$ & $0.00303$\\
& $0.5$  & $0.01169$ & $0.00067$ & $0.00290$\\
& $1$    & $0.01324$ & $0.00076$ & $0.00306$\\
& $2$    & $0.01408$ & $0.00079$ & $0.00461$\\
\hline
\multirow{4}{*}{$r=5$}
& $0.25$ & $0.02693$ & $0.00080$ & $0.00454$\\
& $0.5$  & $0.02811$ & $0.00088$ & $0.00461$\\
& $1$    & $0.03110$ & $0.00091$ & $0.00726$\\
& $2$    & $0.03211$ & $0.00124$ & $0.00732$\\
\hline
\end{tabular}
\end{table}

Table \ref{tab:vector_cost} shows that IMEX-BDF has the smallest per-step cost, IMEX-RK has an intermediate cost, and the rescaled ETDRK schemes are more expensive per step. This ordering is consistent with their respective numbers of internal stages and nonlinear evaluations.

\begin{figure}[!htpb]
\centering
\includegraphics[width=0.96\textwidth]{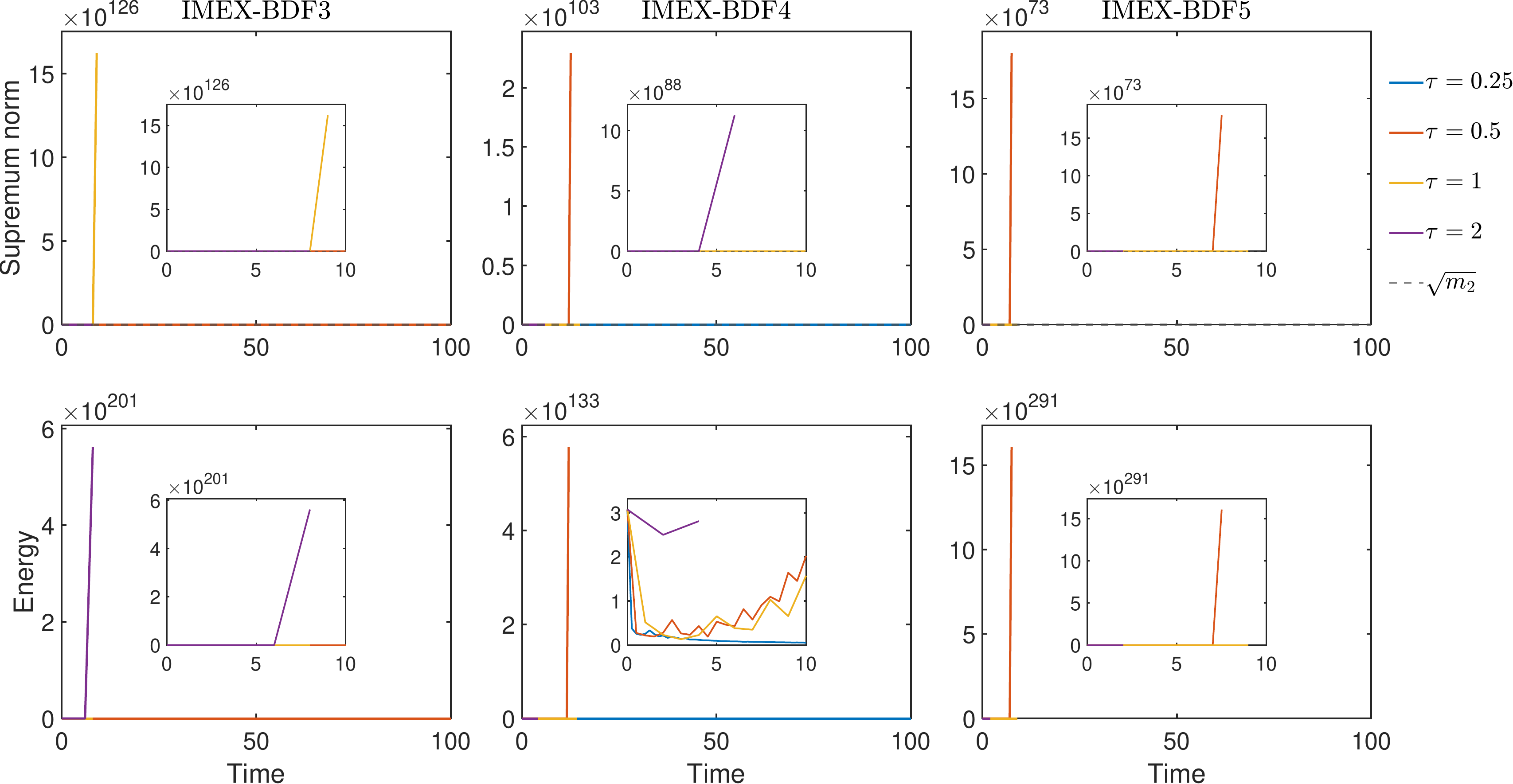}
\caption{Evolution of the supremum norm (upper row) and the discrete energy (lower row) for the IMEX-BDF$r$ schemes, $r=3,4,5$ from left to right, with $\tau=0.25,0.5,1,2$. The insets enlarge the interval $t\in[0,10]$.}
\label{fig:vector-IMEX-BDF}
\end{figure}

\begin{figure}[!htpb]
\centering
\includegraphics[width=0.96\textwidth]{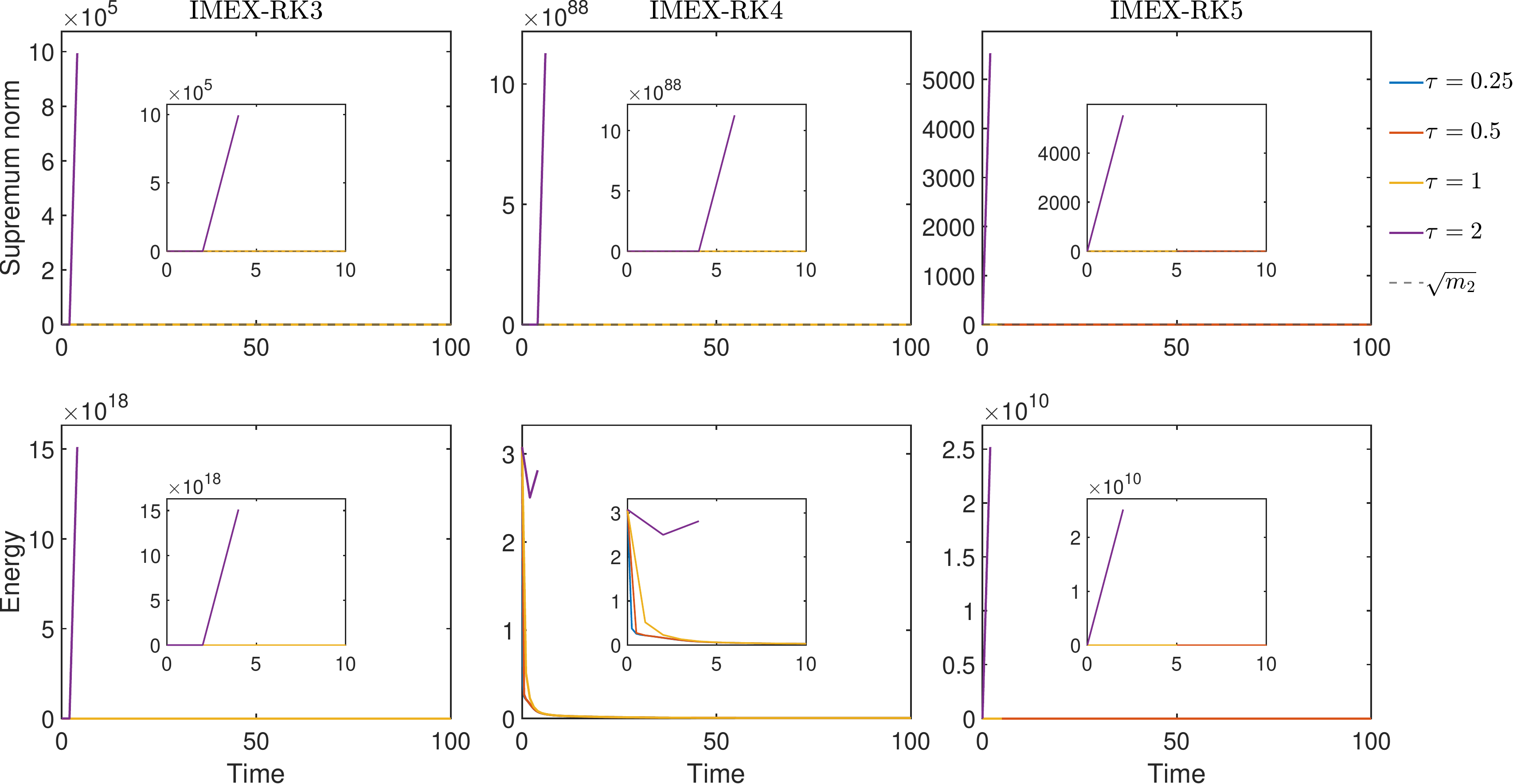}
\caption{Evolution of the supremum norm (upper row) and the discrete energy (lower row) for the IMEX-RK$r$ schemes, $r=3,4,5$ from left to right, with $\tau=0.25,0.5,1,2$. The insets enlarge the interval $t\in[0,10]$.}
\label{fig:vector-IMEX-RK}
\end{figure}

However, the time taken per step alone does not measure whether a method produces a structurally admissible solution at the selected time step.  Figure \ref{fig:4.11} shows that the rescaled ETDRK$r$ scheme satisfies the bound $\|U^n\|_{\mathcal X}\leq1$ for every tested $(r,\tau)$ pair and exhibits monotonically decreasing energy. In contrast, as shown in Figure \ref{fig:vector-IMEX-BDF}--\ref{fig:vector-IMEX-RK}, both the IMEX-BDF and IMEX-RK schemes violate the MBP and energy dissipation properties when $\tau$ is relatively large. Thus, the proposed rescaled ETDRK schemes incur a higher per-step cost in this implementation, but provide substantially stronger robustness of the MBP and energy decay for large time steps.

\subsection{Square-matrix-valued case}
We consider the generalized matrix-valued Allen--Cahn equation \eqref{eq:MAC} with $m_1=m_2=2$ on $\Omega=[-\frac 12,\frac 12]^2$ with periodic boundary conditions. The following petal-shaped initial field consists of two orthogonal matrix branches and therefore satisfies the pointwise bound $\|U_0(x,y)\|_\tF=\sqrt{2}$:
\begin{equation}\label{eq:4.2}
\begin{aligned}
U_0(x,y)=\begin{cases}
\begin{bmatrix}
\cos\alpha&-\sin\alpha\\
\sin\alpha&\cos\alpha
\end{bmatrix}\quad \rho<0.18+0.2\sin(6\theta),\\
\\
\begin{bmatrix}
\cos\alpha&\sin\alpha\\
\sin\alpha&-\cos\alpha
\end{bmatrix}\quad\text{otherwise},\\
\end{cases}
\end{aligned}
\end{equation}
where $\alpha=\alpha(x,y)=\frac{\pi}{2}\sin(2\pi(x+y))$ and $(\rho,\theta)$ denotes the polar coordinate of $(x,y)$, with $\rho=\sqrt{x^2+y^2}$. 

For visualization, we use the determinant $c(t,x,y)=\det(U(t,x,y))$ as the scalar order parameter. The color map represents $c(t,x,y)$, the diffuse interface is identified with the zero-level set $c=0$, and the arrows denote the first column vector of $U$.

\subsubsection{Temporal convergence.}
To test the convergence of the rescaled ETDRK schemes for the square-matrix-valued Allen--Cahn equation, we fix the spatial grid and employ the rescaled ETDRK$r$ schemes ($r=3,4,5$) with time steps $\tau=0.1\times 2^{-k}$, $k=0,\ldots,4$. The $L^2$- and $L^\infty$-errors at a fixed final time (relative to a reference solution computed with a very small time step) and the corresponding convergence rates are presented in Table \ref{tab:2}. It is evident that whenever $\tau$ is halved, the errors decay at a rate $O(\tau^r)$, and the observed convergence orders in both norms are close to the designed orders $r$. This confirms that the rescaled ETDRK$r$ schemes retain their optimal accuracy for the square-matrix-valued Allen--Cahn equation.
\begin{table}[!htpb]
\centering
\normalsize
\caption{$L^2$- and $L^\infty$-errors and convergence rates of the rescaled ETDRK$r$ schemes ($r=3,4,5$) for the square-matrix-valued Allen--Cahn equation with initial condition \eqref{eq:4.2}.}
\label{tab:2}
\vspace{10pt}
\begin{tabular}{cccccc}
\hline
$r$&$\tau$ & $L^2$ error & Rate&$L^{\infty}$ error & Rate \\
\hline
\multirow{5}{*}{$r=3$}&
$0.1\times 2^{0}$&$4.24\times 10^{-3}$&-&$8.75\times 10^{-3}$&-\\
&$0.1\times 2^{-1}$&$6.50\times 10^{-4}$&2.70&$1.36\times 10^{-3}$&2.68\\
&$0.1\times 2^{-2}$&$9.03\times 10^{-5}$&2.85&$1.91\times 10^{-4}$&2.84\\
&$0.1\times 2^{-3}$&$1.19\times 10^{-5}$&2.92&$2.53\times 10^{-5}$&2.92\\
&$0.1\times 2^{-4}$&$1.53\times 10^{-6}$&2.96&$3.25\times 10^{-6}$&2.96\\
\hline
\multirow{5}{*}{$r=4$}&
$0.1\times 2^{0}$&$4.39\times 10^{-4}$&-&$9.70\times 10^{-4}$&-\\
&$0.1\times 2^{-1}$&$3.45\times 10^{-5}$&3.67&$7.74\times 10^{-5}$&3.65\\
&$0.1\times 2^{-2}$&$2.42\times 10^{-6}$&3.83&$5.48\times 10^{-6}$&3.82\\
&$0.1\times 2^{-3}$&$1.61\times 10^{-7}$&3.91&$3.65\times 10^{-7}$&3.91\\
&$0.1\times 2^{-4}$&$1.04\times 10^{-8}$&3.96&$2.35\times 10^{-8}$&3.95\\
\hline
\multirow{5}{*}{$r=5$}&
$0.1\times 2^{0}$&$3.90\times 10^{-5}$&-&$9.09\times 10^{-5}$&-\\
&$0.1\times 2^{-1}$&$1.55\times 10^{-6}$&4.65&$3.66\times 10^{-6}$&4.63\\
&$0.1\times 2^{-2}$&$5.47\times 10^{-8}$&4.82&$1.30\times 10^{-7}$&4.81\\
&$0.1\times 2^{-3}$&$1.82\times 10^{-9}$&4.91&$4.35\times 10^{-9}$&4.91\\
&$0.1\times 2^{-4}$&$1.28\times 10^{-10}$&3.83&$1.42\times 10^{-10}$&4.94\\
\hline
\end{tabular}
\end{table}

\subsubsection{Time-step robustness of the MBP and energy.}
We examine the discrete MBP and energy dissipation of the rescaled ETDRK$r$ schemes ($r=3,4,5$) for the square-matrix-valued case. The terminal time is set to $T=100$. The time steps are chosen as $\tau=0.25,0.5,1,2$. Figure \ref{fig:4.21} plots the time evolution of $\|U\|_\mX$ (upper row) and the discrete energy (lower row). In all tests, $\|U^n\|_\mX$ remains below the theoretical bound $\sqrt{2}$, and the discrete energy decreases monotonically. These results confirm the MBP preservation and the energy-dissipative behavior of the proposed rescaled ETDRK schemes in the square-matrix-valued case.
\begin{figure}[!htpb]
\centering
\includegraphics[width=0.96\textwidth]{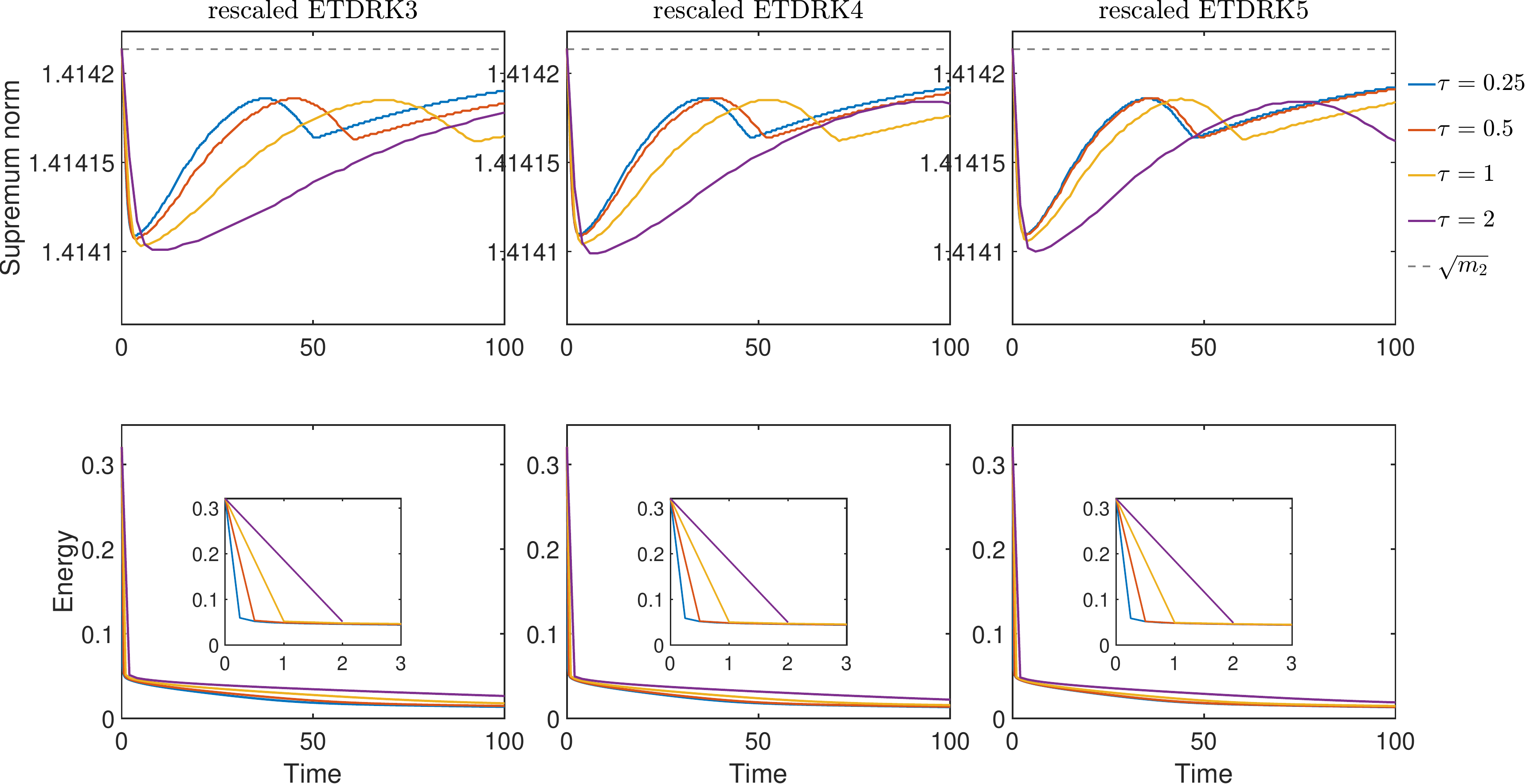}
\caption{Evolution of the supremum norm $\|U\|_{\mathcal{X}}$ (upper row) and the discrete energy (lower row) for the rescaled ETDRK$r$ schemes ($r=3,4,5$) with initial condition \eqref{eq:4.2}. In each row, the three panels correspond to $r=3,4,5$ from left to right.}
\label{fig:4.21}
\end{figure}

\subsubsection{Evolution of the square-matrix-valued field.}
We next show the interfacial evolution of the square-matrix-valued field for $\varepsilon=0.01$. The computation is performed for the initial condition \eqref{eq:4.2}. Figure~\ref{fig:4.21-evolution} shows that the initial petal-shaped region is rapidly regularized by the diffuse-interface dynamics. The highly curved tips disappear first, and the interface then becomes smoother and more circular. At later times, the enclosed region shrinks and eventually disappears, while the matrix field relaxes toward a nearly uniform orthogonal phase. This behavior is consistent with a diffuse-interface mean-curvature-flow mechanism. The interface regions with larger curvature move faster, while the small enclosed regions tend to shrink.
\begin{figure}[!htpb]
\centering
\includegraphics[width=0.98\textwidth]{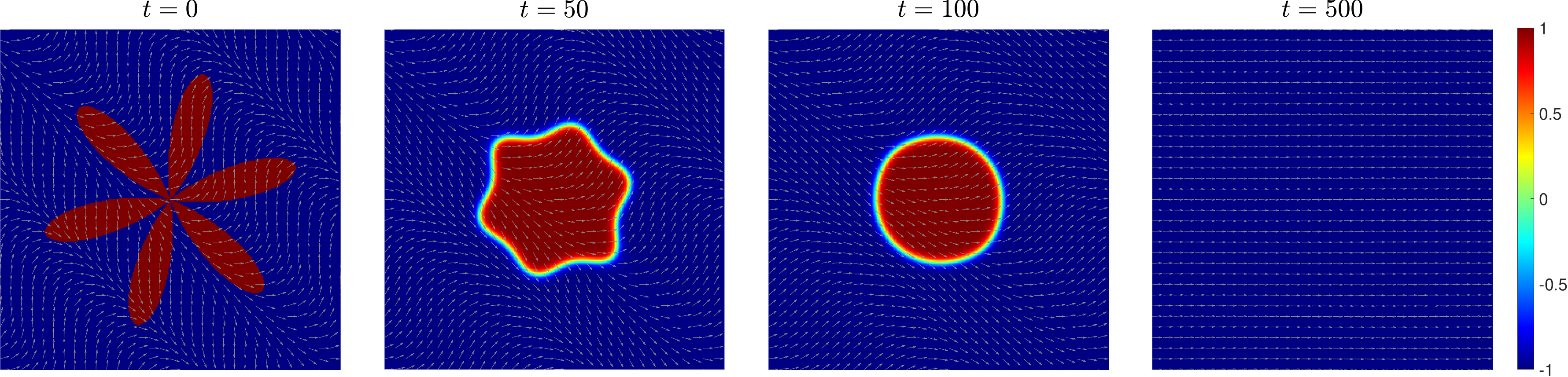}
\caption{Evolution of the square-matrix-valued field for $\varepsilon=0.01$ at $t=0,50,100,500$.}
\label{fig:4.21-evolution}
\end{figure}

\subsubsection{Sensitivity to the thickness parameter $\varepsilon$.}

We study the influence of the thickness parameter $\varepsilon$ on the interfacial evolution, in connection with the sharp-interface behavior discussed in \cite{fei2023matrix}. We use the same initial condition \eqref{eq:4.2} and run the rescaled ETDRK$5$ scheme to compare the cases for $\varepsilon = 0.01, 0.05,$ and $0.1$. Since the interfacial time scale changes with $\varepsilon$, different snapshot times are used, and the results are plotted in Figure \ref{fig:4.21-epsilon}. The figure shows that the thickness parameter strongly affects both the width of the transition layer and the relaxation speed. For $\varepsilon=0.01$, the transition layer remains thin, and the two orthogonal phases are clearly separated. The initial petal-shaped interface is gradually smoothed near the highly curved tips, but the petal structure is still visible at $t=10$. For $\varepsilon=0.05$, the transition layer becomes wider and the red phase is regularized much faster. By $t=10$, the petal-shaped region has essentially disappeared. For $\varepsilon=0.1$, the smoothing is even stronger: the interface is already very diffuse at early times, and the solution rapidly relaxes toward a nearly uniform phase. Thus, increasing $\varepsilon$ produces a thicker diffuse interface and accelerates the visible disappearance of the initial pattern. These observations are consistent with the sharp-interface picture in \cite{fei2023matrix}. For small $\varepsilon$, the diffuse interface remains localized and moves in a curvature-driven manner, while larger $\varepsilon$ leads to stronger diffuse smoothing. In all tested cases, the first-column vector field evolves smoothly and no visible spurious oscillations appear, indicating that the proposed rescaled ETDRK scheme remains stable with respect to the thickness parameter.

\begin{figure}[!htpb]
\centering
\subfigure[$\varepsilon=0.01$]{\includegraphics[width=0.98\textwidth]{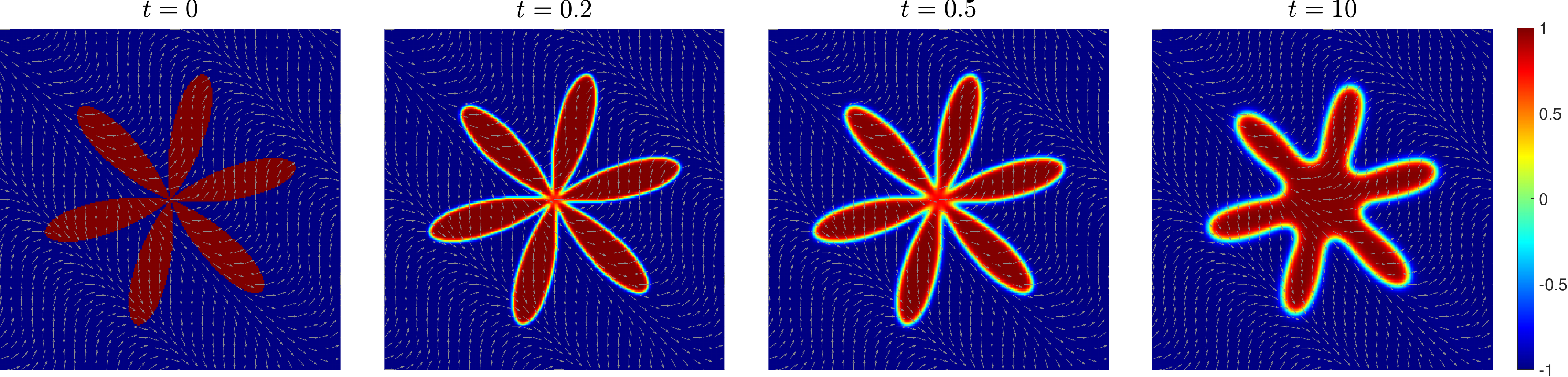}}\\[-1mm]
\subfigure[$\varepsilon=0.05$]{\includegraphics[width=0.98\textwidth]{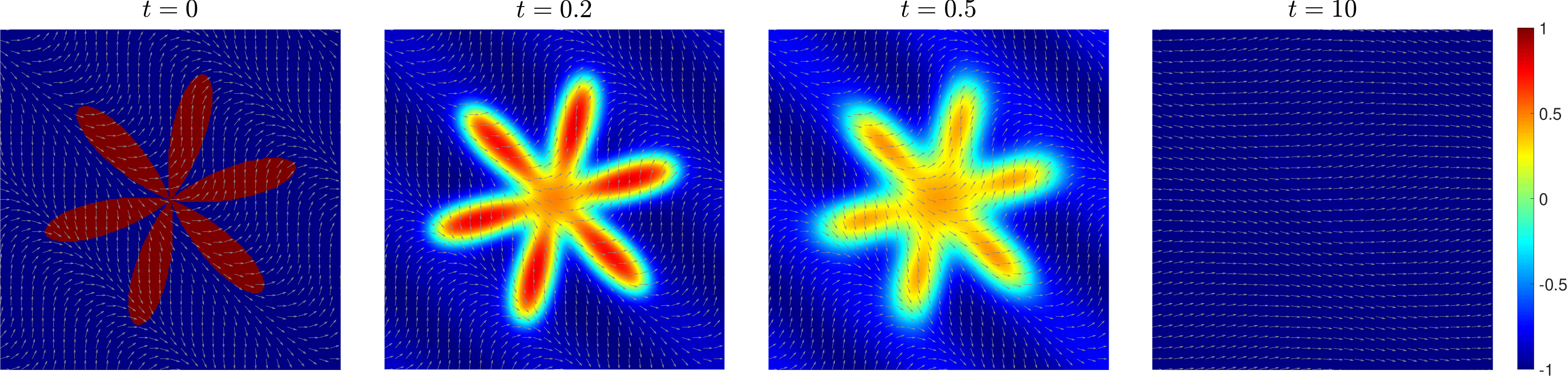}}\\[-1mm]
\subfigure[$\varepsilon=0.1$]{\includegraphics[width=0.98\textwidth]{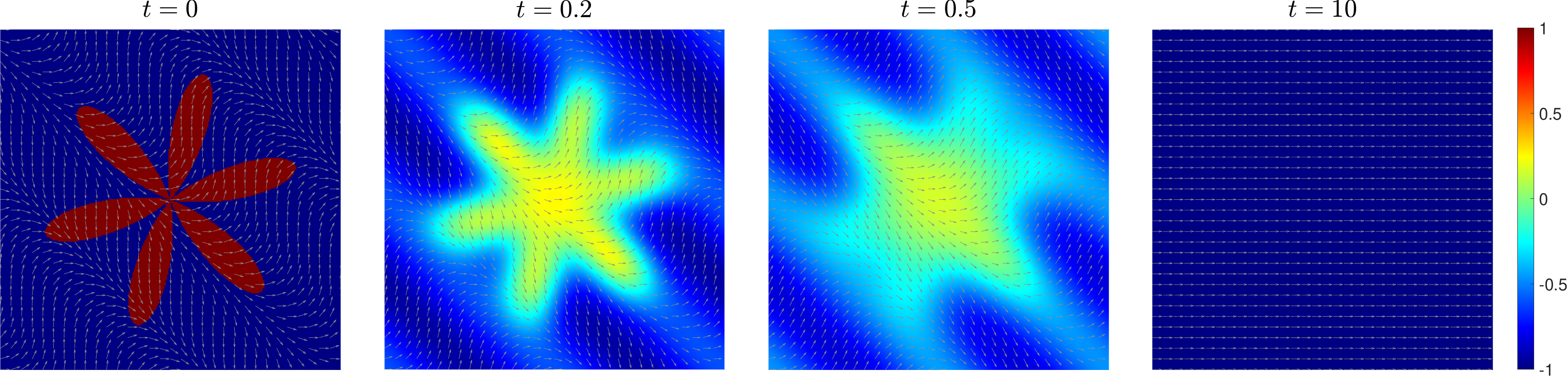}}
\caption{Evolution of the square-matrix-valued field for different thickness parameters. From top to bottom, the rows correspond to $\varepsilon=0.01$, $0.05$, and $0.1$, respectively. The snapshots are taken at $t=0,0.2,0.5,10$.}
\label{fig:4.21-epsilon}
\end{figure}

\subsubsection{Denoising of orthogonal matrix fields.}
We further examine a denoising experiment for an orthogonal matrix field. Let $U^0_0$ be the clean initial condition in \eqref{eq:4.2}. We generate a noisy matrix field by adding componentwise Gaussian perturbations,
\begin{equation*}
\widetilde U^\sigma_0(x,y)=U^0_0(x,y)+\sigma G(x,y),\qquad \sigma=0.15,
\end{equation*}
where the entries of $G$ are independent standard normal random variables. Since the discrete MBP theorem for the proposed rescaled ETDRK$r$ schemes requires the initial data to lie in the admissible ball, we project the perturbed field pointwise onto this ball and define
\begin{equation*}
U^\sigma_0(x,y)=
\min\left\{1,\frac{\sqrt{2}}{\|\widetilde U^\sigma_0(x,y)\|_\tF}\right\}
\widetilde U^\sigma_0(x,y),
\end{equation*}
where $\|U^\sigma_0(x,y)\|_\tF\leq\sqrt{2}$ for all grid points. We evolve both the clean initial field $U^0_0$ and the noisy initial field $U^\sigma_0$ by the rescaled ETDRK$5$ scheme with $\tau=0.01$, and denote the two numerical trajectories by $U^0(t)$ and $U^\sigma(t)$, respectively. To quantify the effect of denoising, we use the following error
\begin{equation*}
e_{\rm traj}(t)=\|U^\sigma(t)-U^0(t)\|_{L^2}.
\end{equation*}
Here $e_{\rm traj}$ compares the noisy evolution with the clean evolution at the same time.

Figure \ref{fig:4.22-noise-fields} compares the clean and noisy evolutions at $t=0,0.1,0.5,1$, both computed by the proposed rescaled ETDRK$5$ scheme with $\tau=0.01$. The initially noisy field is rapidly regularized and becomes visually close to the clean evolution. This demonstrates the smoothing effect of the matrix-valued Allen--Cahn flow and shows that the proposed scheme effectively captures the corresponding denoising dynamics. The short-time diagnostic curves are presented in Figure \ref{fig:4.23-noise-short-metrics}. The trajectory error $e_{\rm traj}(t)$ decreases from $2.70\times10^{-1}$ at $t=0$ to $1.69\times10^{-2}$ at $t=1$, confirming that the noisy evolution approaches the clean evolution.
\begin{figure}[!htpb]
\centering
\subfigure[Clean evolution]{\includegraphics[width=0.98\textwidth]{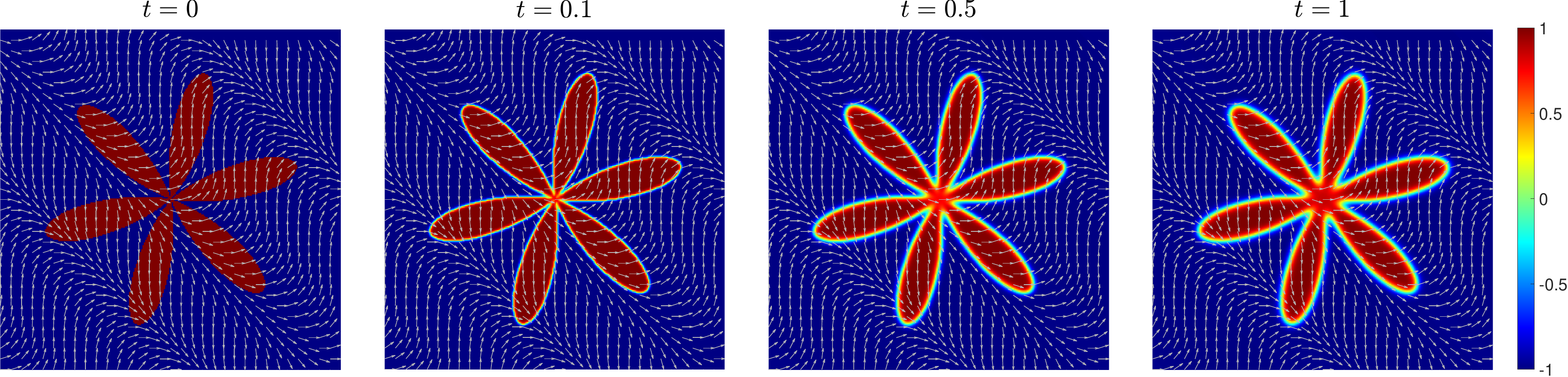}}\\[-1mm]
\subfigure[Noisy evolution]{\includegraphics[width=0.98\textwidth]{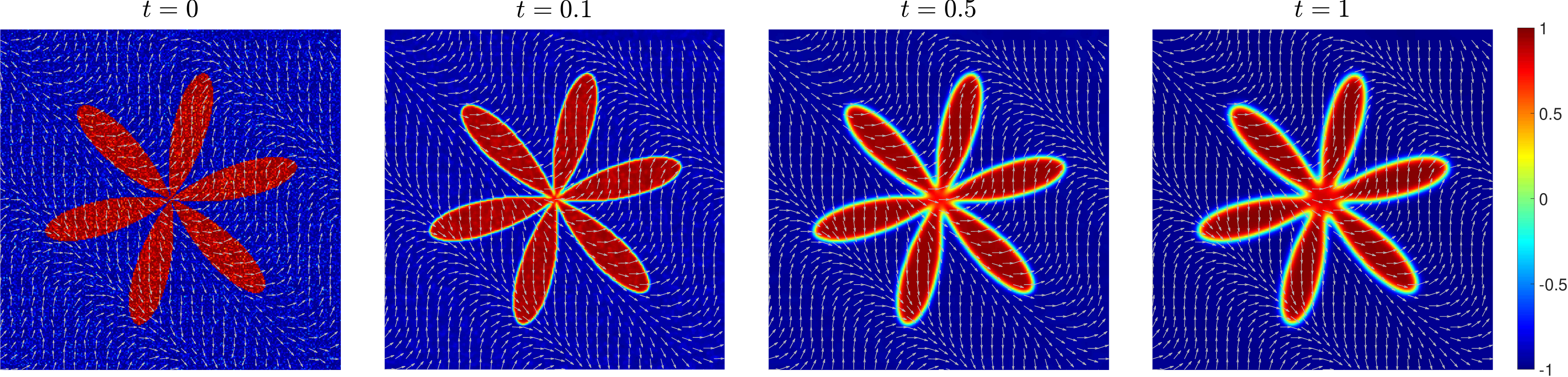}}
\caption{Denoising experiment for the square-matrix-valued Allen--Cahn equation. The upper row shows the evolution from the clean initial field $U^0(0)$, and the lower row shows the evolution from the noisy initial field $U^\sigma(0)$ with noise level $\sigma=0.15$.}
\label{fig:4.22-noise-fields}
\end{figure}
\begin{figure}[!htpb]
\centering
\includegraphics[width=0.32\textwidth]{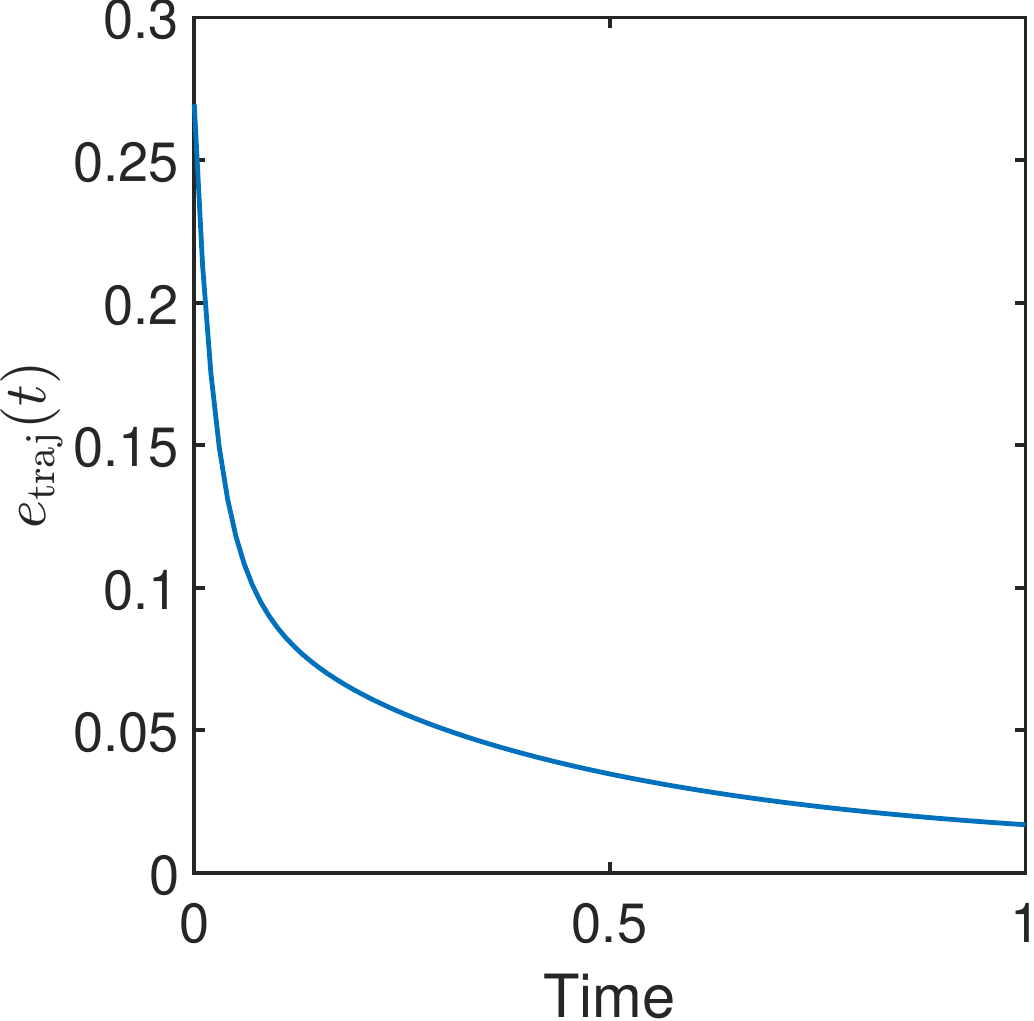}
\caption{The diagnostic curve for the denoising experiment on $t\in[0,1]$.}
\label{fig:4.23-noise-short-metrics}
\end{figure}

To examine structure preservation for noisy data, we further run the rescaled ETDRK$5$ scheme up to $T=100$ with $\tau=0.1,0.5,1,2$. As shown in Figure \ref{fig:4.24-noise-long-structure}, the supremum norm remains below the theoretical bound $\sqrt{2}$ for all tested time steps, while the discrete energy decreases monotonically. These results demonstrate that the proposed scheme captures the short-time denoising dynamics while preserving the discrete MBP and energy-dissipation property, even for noisy matrix-valued data and large time steps.
\begin{figure}[!htpb]
\centering
\includegraphics[width=0.7\textwidth]{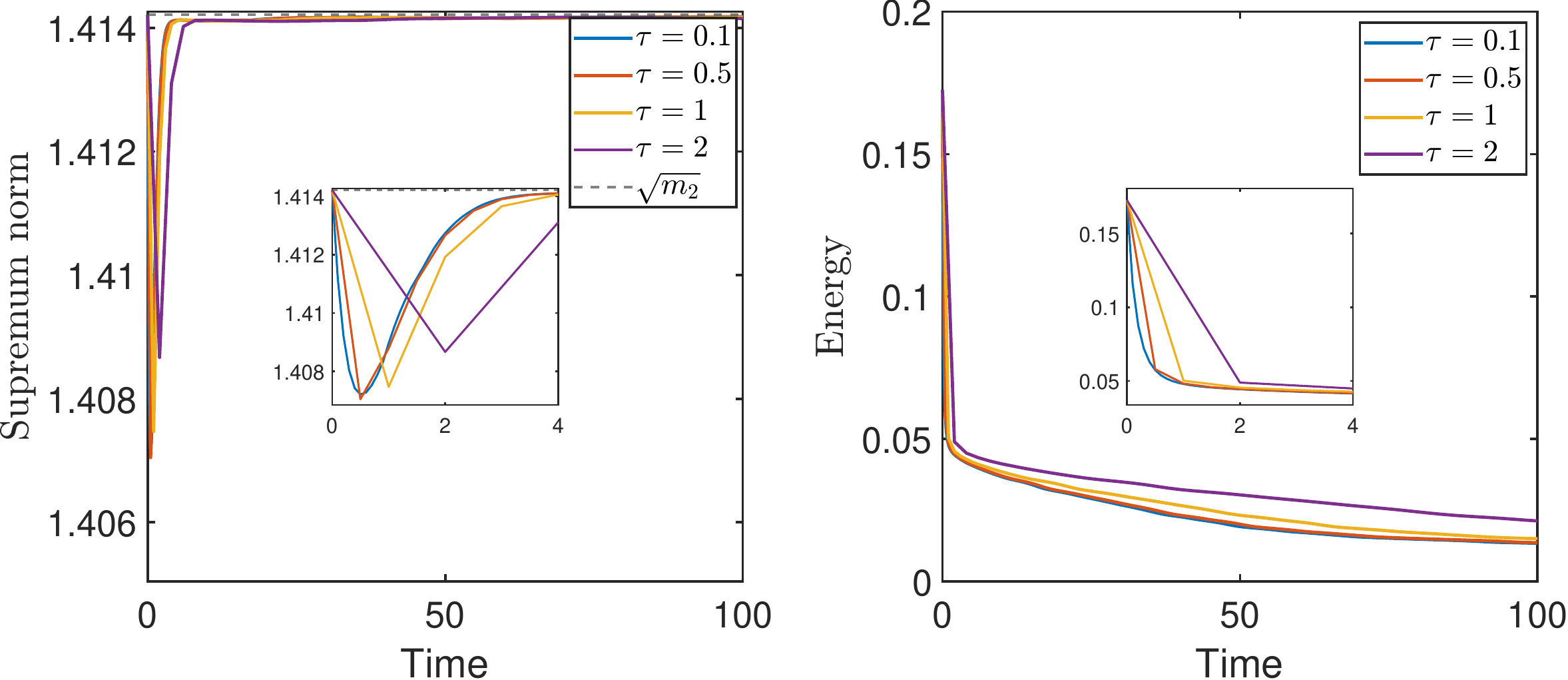}
\caption{The structure-preserving test for the noisy initial matrix field. The rescaled ETDRK$5$ scheme is run up to $T=100$ with $\tau=0.1,0.5,1,2$.}
\label{fig:4.24-noise-long-structure}
\end{figure}

\subsection{Generalized matrix-valued case}
Let $\Omega=[-\frac12,\frac12]^2$ and choose $K\ge2$ distinct seed points $p_k=(x_k,y_k)\in\Omega$ for $k=1,\dots,K$. The associated Voronoi tessellation $\{\Omega_k\}_{k=1}^K$ of $\Omega$ is defined by
$$\Omega_k=\Big\{(x,y)\in\Omega:\ |(x,y)-p_k|\le |(x,y)-p_\ell|\ \ \forall\,\ell\neq k\Big\},\qquad k=1,\dots,K,$$
which yields a polycrystalline configuration with piecewise-constant grain labels \cite{staublin2022phase}. For each grain $\Omega_k$, we prescribe an angle  $\alpha_k\in[0,2\pi)$ and a sign $s_k\in\{+1,-1\}$ and define $\beta(x,y)=\frac{\pi}{20}\sin\big(2\pi(x-y)\big)$. Then, for $k=1,\dots,K$, we define a $3\times2$ matrix field $U_k:\Omega\to\mathbb{R}^{3\times2}$ by
\begin{equation*}
U_k(x,y)=
\begin{pmatrix}
\cos(\alpha_k)\cos(\beta(x,y)) & s_k\big(-\sin(\alpha_k)\big)\\
\sin(\alpha_k)\cos(\beta(x,y)) & s_k\cos(\alpha_k)\\
\sin(\beta(x,y)) & 0
\end{pmatrix}.   
\end{equation*}
Equivalently, $U_k=[u_{1,k}\ u_{2,k}]$ with
$$u_{1,k}(x,y)=\big(\cos\alpha_k\cos\beta,\sin\alpha_k\cos\beta,\sin\beta\big)^\top,
\quad u_{2,k}(x,y)=s_k\big(-\sin\alpha_k,\ \cos\alpha_k,\ 0\big)^\top,$$
so that $u_{1,k}$ and $u_{2,k}$ are orthonormal for all $(x,y)\in\Omega$, i.e.,
$$U_k(x,y)^\top U_k(x,y)=I_2,\quad \forall (x,y)\in\Omega,\ \forall k=1,\dots,K.$$
In particular, the third component $U_{k,31}(x,y)=\sin(\beta(x,y))$ encodes an out-of-plane tilt, which does not exist in the $2\times2$ case :
\begin{equation*}
U_k(x,y)=
\begin{pmatrix}
\cos(\alpha_k)\cos(\beta(x,y)) & s_k\big(-\sin(\alpha_k)\big)\\
\sin(\alpha_k)\cos(\beta(x,y)) & s_k\cos(\alpha_k)
\end{pmatrix}.
\end{equation*}
Then, the initial condition is prescribed by the piecewise definition
\begin{equation}\label{eq:4.3}
U_0(x,y)=
\begin{cases}
U_1(x,y), & (x,y)\in\Omega_1,\\
U_2(x,y), & (x,y)\in\Omega_2,\\
\quad \vdots & \quad \vdots\\
U_K(x,y), & (x,y)\in\Omega_K.
\end{cases}
\end{equation}
For visualization, we introduce the projected determinant
\begin{equation*}
c(t,x,y):=\det\big(U_{1:2,1:2}(t,x,y)\big)=(u_1(t,x,y)\times u_2(t,x,y))\cdot e_3 \in [-1,1],
\end{equation*}
where $u_1$ and $u_2$ denote the first and second columns of $U$, respectively. 
We use $c$ as an order parameter, and the diffuse interface is defined by its zero-level set $\{c(t,x,y)=0\}$. The simulations are carried out by the rescaled ETDRK$r$ scheme with $r=3,4,5$. We investigate the MBP preservation and discrete energy dissipation of the proposed method, and compare the interfacial dynamics of the $3\times2$ matrix-valued field with those of the $2\times2$ matrix-valued case.

\subsubsection{Time-step robustness of the MBP and energy.}
We first test the structure-preserving behavior of the rescaled ETDRK$r$ schemes ($r=3,4,5$) with $\epsilon=0.01$ and $\kappa=7$. The final time is set to $T=500$, and four time steps $\tau=0.25,0.5,1,2$ are used for each order. Figure \ref{fig:4.31} shows that the supremum norm stays below the theoretical bound $\sqrt{2}$ for all tested time steps and all three high-order schemes. The discrete energy also decreases monotonically up to $T=500$, including for the large step size $\tau=2$. These results provide numerical evidence for the unconditional MBP preservation and energy stability of the proposed rescaled ETDRK schemes in the generalized $3\times2$ matrix-valued case.

\begin{figure}[!htpb]
\centering
\includegraphics[width=0.96\textwidth]{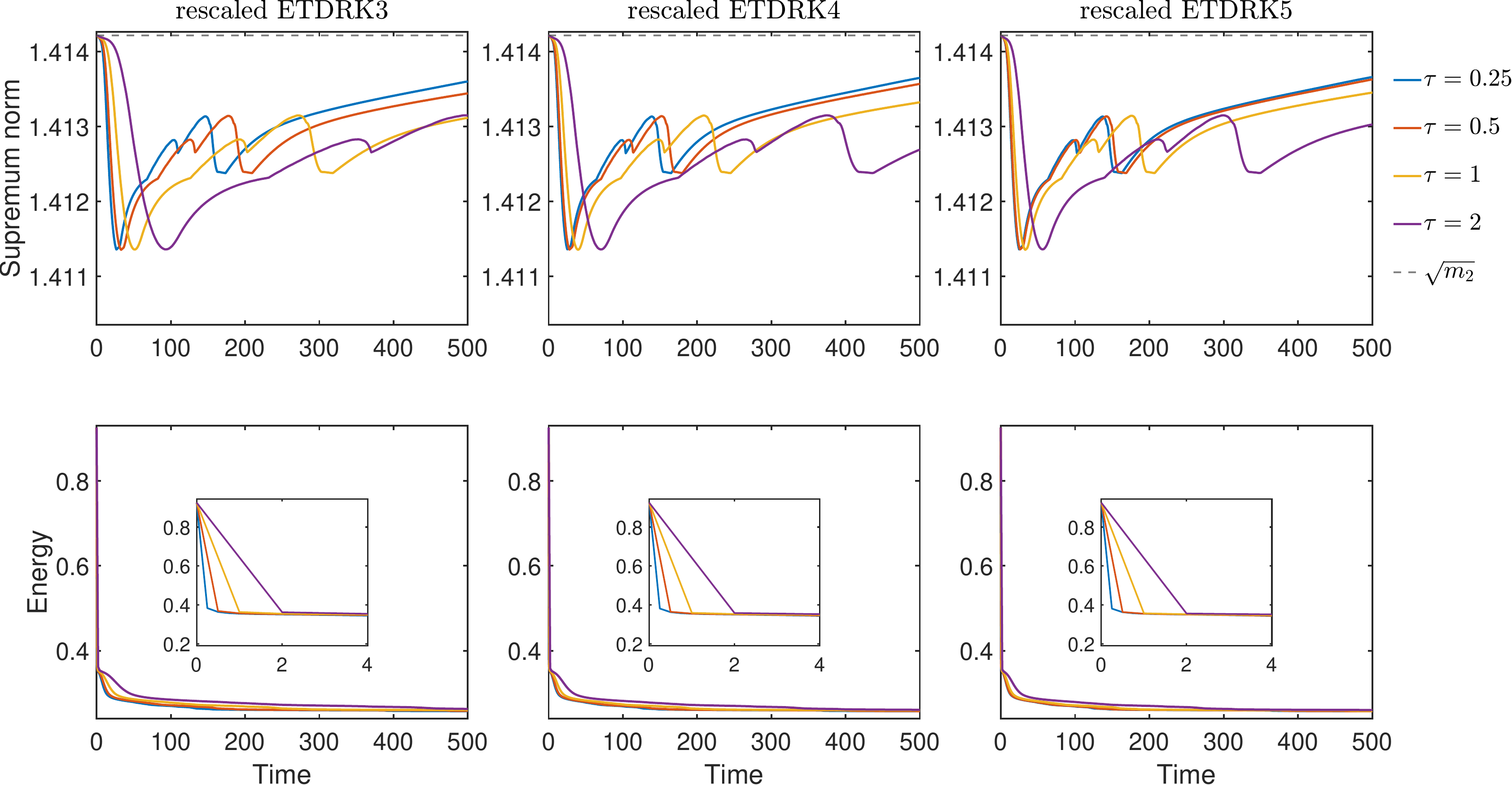}
\caption{Evolution of the supremum norm $\|\cdot\|_{\mathcal{X}}$ (upper row) and the discrete energy (lower row) for the rescaled ETDRK$r$ schemes ($r=3,4,5$) with initial condition \eqref{eq:4.3}. The three columns correspond to $r=3,4,5$, respectively.}
\label{fig:4.31}
\end{figure}

\subsubsection{Interfacial dynamics and the out-of-plane component.}
We next run the rescaled ETDRK$5$ scheme with $\tau=1$ and record the interfacial snapshots at $t=0,40,80,200$. In addition, we also show the out-of-plane component $U_{31}$ and its supremum norm up to $T=500$.
\begin{figure}[!htpb]
\centering
\includegraphics[width=0.98\textwidth]{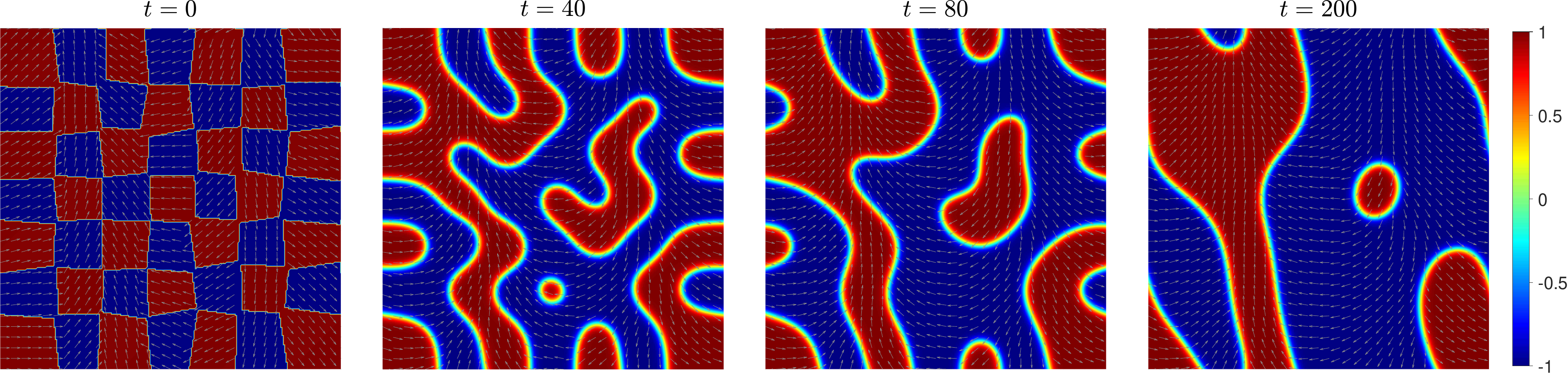}
\caption{Evolution of the $2\times2$ matrix-valued field and interface at $t=0,40,80,200$.}
\label{fig:4.32}
\end{figure}
\begin{figure}[!htpb]
\centering
\includegraphics[width=0.98\textwidth]{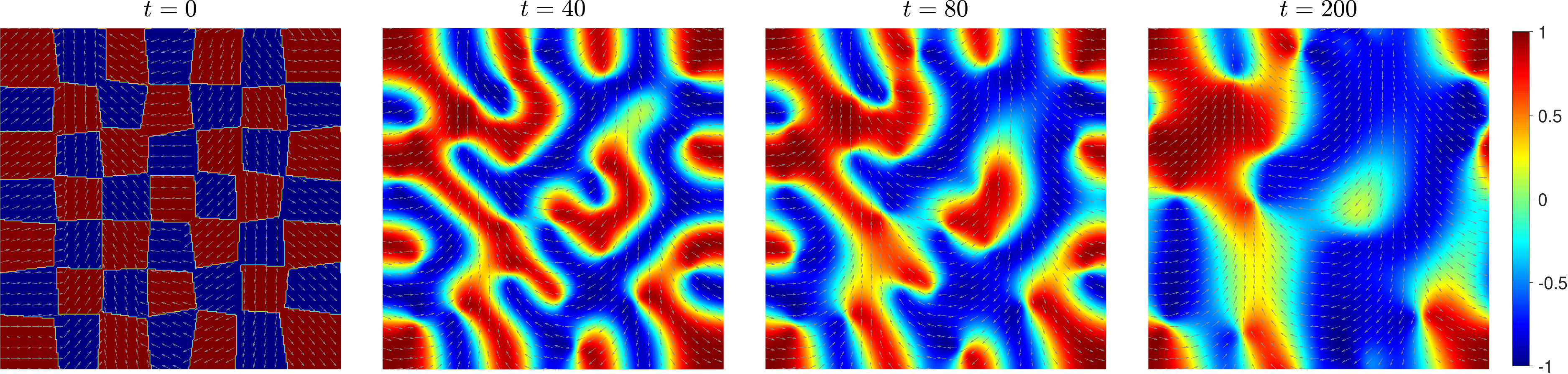}
\caption{Evolution of the $3\times2$ matrix-valued field and interface at $t=0,40,80,200$.}
\label{fig:4.33}
\end{figure}
\begin{figure}[!htpb]
\centering
\includegraphics[width=0.98\textwidth]{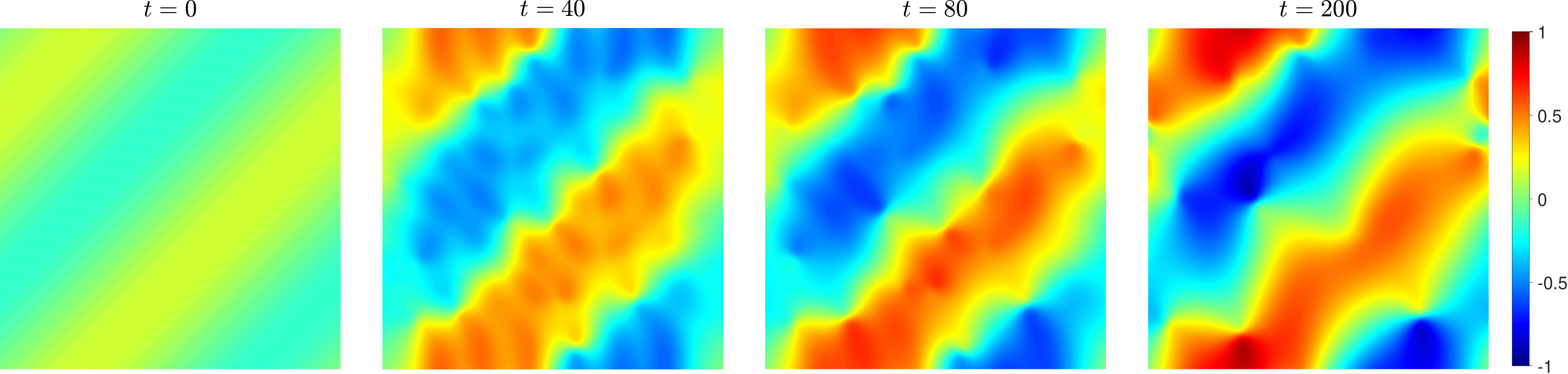}
\caption{The value of $U_{31}$ at $t=0,40,80,200$.}
\label{fig:4.34}
\end{figure}

Figures \ref{fig:4.32}--\ref{fig:4.33} compare the interface evolution of the $2\times2$ and $3\times2$ matrix-valued fields generated from Voronoi-type polycrystalline initial configurations on the domain $\Omega=[-\frac12,\frac12]^2$. In both cases, the diffuse interface is visualized by the order parameter $c(t,x,y)$, and the interface is defined as the zero-level set $c=0$. The sharp jumps between neighboring blocks at $t=0$ are quickly smoothed into transition layers; subsequently, small domains shrink and vanish while larger regions expand. The main difference is that in the $2\times2$ case the dynamics is governed by matrix variations, whereas in the $3\times2$ case the additional component $U_{31}$ evolves noticeably in time, as shown in Figure \ref{fig:4.34}. This degree of freedom changes the distribution of $c$ and leads to interface shapes that differ from those in the square-matrix-valued case.

\begin{figure}[!htpb]
\centering
\includegraphics[width=0.30\textwidth]{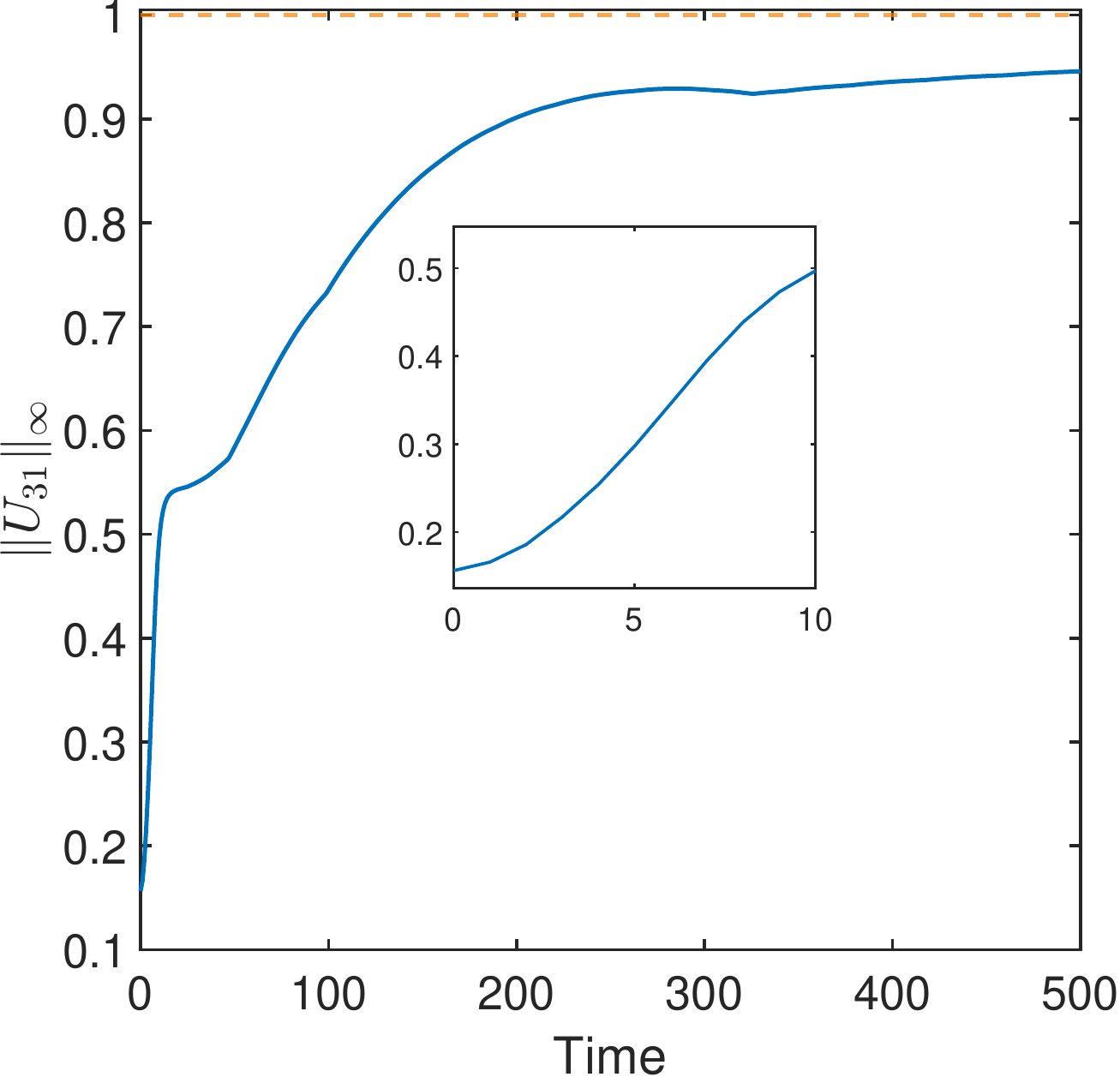}
\caption{Evolution of $\|U_{31}\|_\infty$ over time for the $3\times2$ matrix-valued case.}
\label{fig:4.35}
\end{figure}

Figure \ref{fig:4.35} plots the time evolution of $\|U_{31}\|_\infty$. The quantity increases rapidly at the beginning and then grows more slowly, reaching approximately $0.94577$ at $T=500$ while remaining below $1$. For $U=[u_1,u_2]\in\mathbb{R}^{3\times2}$, when $U^\top U = I_2$ and $U_{32}=0$, we have
\begin{equation*}
c = (u_1 \times u_2) \cdot e_3 = s \sqrt{1 - U_{31}^2},\qquad s\in\{+1,-1\}.
\end{equation*}
Therefore, a larger $|U_{31}|$ leads to a smaller $|c|$. In the $2\times2$ case, $U_{31}=0$, and under this condition, $c$ rapidly approaches $\pm1$ within each region. Conversely, in the $3\times2$ simulation, the growth of $U_{31}$ weakens this tendency, so the values of $c$ in certain regions are not forced to remain close to $\pm1$. This leads to a distinct evolution of the interface.

\subsection{Three-dimensional matrix-valued case}
We consider a fully three-dimensional matrix-valued problem with $(m_1,m_2)=(3,3)$ on $\Omega=[-\frac12,\frac12]^3$ and periodic boundary conditions. We take $\varepsilon=0.01$ and $\kappa=10$. Let $A=\{(x,y,z)\in\Omega:\ |x|+|y|+|z|\leq 0.34\}$. The initial value is given by
\begin{equation}\label{eq:4.5}
\begin{aligned}
U_0(x,y,z)=
\begin{cases}
\mathcal{P}\begin{bmatrix}
\frac{1}{2}\cos\alpha&\frac{\sqrt{6}}{2}\cos\alpha&-\frac{\sqrt{2}}{2}\sin\alpha\\
\frac{1}{2}\sin\alpha&\frac{\sqrt{6}}{2}\sin\alpha&\frac{\sqrt{2}}{2}\cos\alpha\\
\frac{\sqrt{3}}{2}&-\frac{\sqrt{2}}{2}&0
\end{bmatrix}\quad \mbox{if}~$(x,y,z)$~\mbox{satisfy condition A},\\
\\
\mathcal{P}\begin{bmatrix}
-\frac{1}{2}\cos\alpha&\frac{\sqrt{6}}{2}\cos\alpha&\frac{\sqrt{2}}{2}\sin\alpha\\
-\frac{1}{2}\sin\alpha&\frac{\sqrt{6}}{2}\sin\alpha&-\frac{\sqrt{2}}{2}\cos\alpha\\
\frac{\sqrt{3}}{2}&\frac{\sqrt{2}}{2}&0
\end{bmatrix}\quad\text{otherwise},\\
\end{cases}
\end{aligned}
\end{equation}
where $\alpha(x,y,z)=4\pi xyz$ and $\mathcal P$ denotes the pointwise projection onto the Frobenius ball $\{U\in\mathbb R^{3\times3}:\|U\|_\tF\leq\sqrt3\}$. This construction gives a double-pyramid interface at $t=0$ and satisfies $\|U_0(x,y,z)\|_\tF\leq\sqrt3$ at every grid point.

\subsubsection{Time-step robustness of the MBP and energy.}
We first test the three-dimensional structure-preserving behavior of the rescaled ETDRK$r$ schemes with $r=3,4,5$. The spatial domain is discretized by a $32\times32\times32$ grid, and the time steps are chosen as $\tau=0.25,0.5,1,2$. Figure \ref{fig:4.41} shows the evolution of the supremum norm and the discrete energy. For all three orders and all tested time steps, the numerical solution satisfies $\|U^n\|_\mX\leq\sqrt3$, and the discrete energy decreases monotonically over the whole time interval. These results indicate that the proposed rescaled ETDRK schemes retain their structure-preserving behavior in the three-dimensional setting, even for relatively large time steps.

\begin{figure}[!htpb]
\centering
\includegraphics[width=0.96\textwidth]{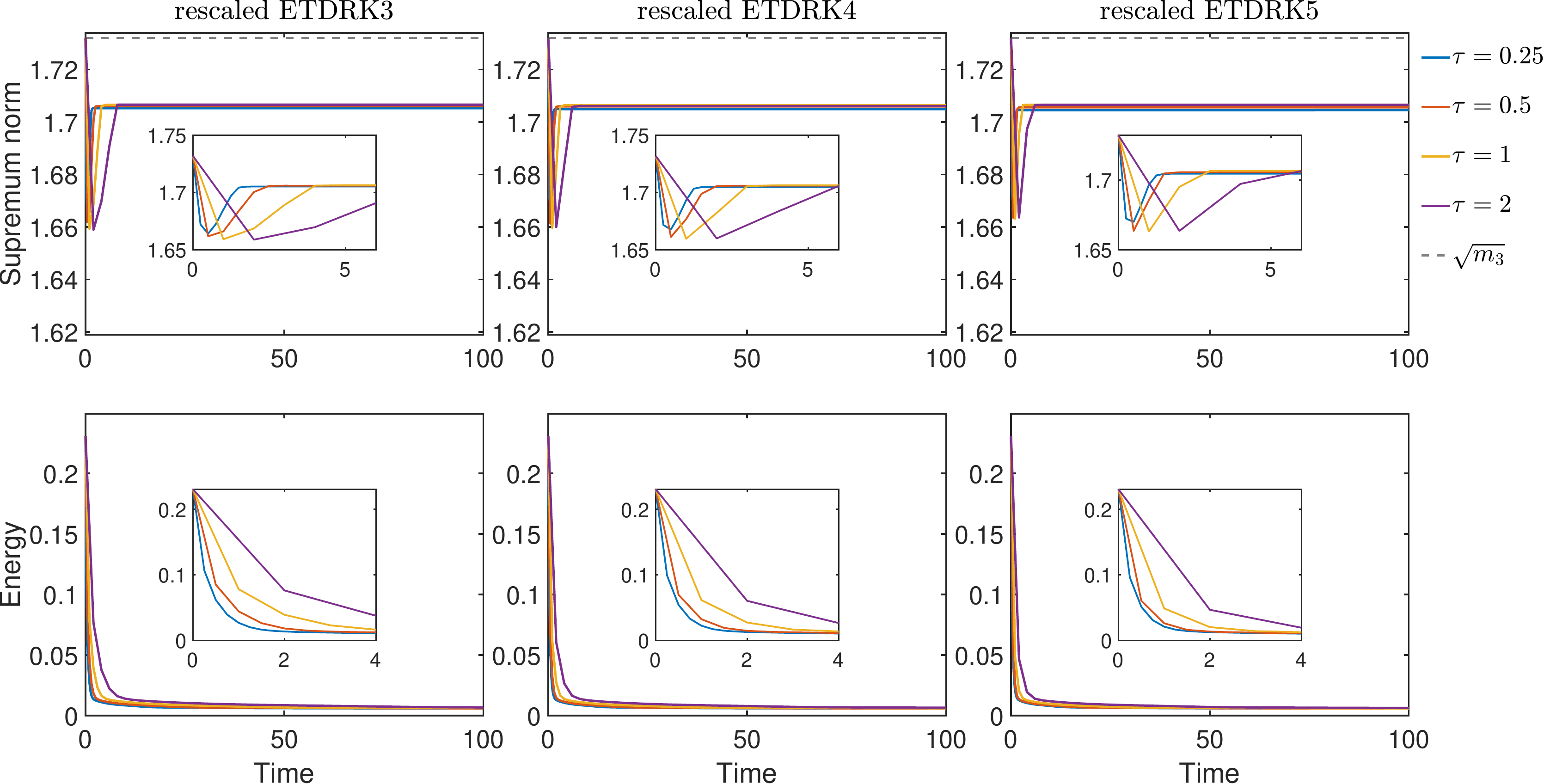}
\caption{Evolution of the supremum norm $\|U\|_{\mathcal X}$ (upper row) and the discrete energy (lower row) for the three-dimensional matrix-valued problem with initial condition \eqref{eq:4.5}. The three columns correspond to the rescaled ETDRK$r$ schemes with $r=3,4,5$.}
\label{fig:4.41}
\end{figure}

\subsubsection{Three-dimensional interfacial evolution.}
We next illustrate the three-dimensional interfacial dynamics by using the rescaled ETDRK$5$ scheme with $\tau=2$ up to $T=500$. The snapshots are recorded at $t=0,150,300,500$. Figure \ref{fig:4.42} displays the zero-level surface $\det(U)=0$ from both an oblique three-dimensional view and a top view, together with the first-column vector field of $U$. The initial zero-level surface has a double-pyramid shape. During the evolution, the sharp edges and vertices are quickly smoothed by the diffuse-interface dynamics. Then, the surface becomes progressively rounder and shrinks in a stable manner. The vector field also becomes more coherent. No visible oscillations or artificial irregularities appear in the interface, which further supports the robustness of the proposed method for three-dimensional matrix-valued simulations.

\begin{figure}[!htpb]
\centering
\subfigure[Three-dimensional view]{\includegraphics[width=0.98\textwidth]{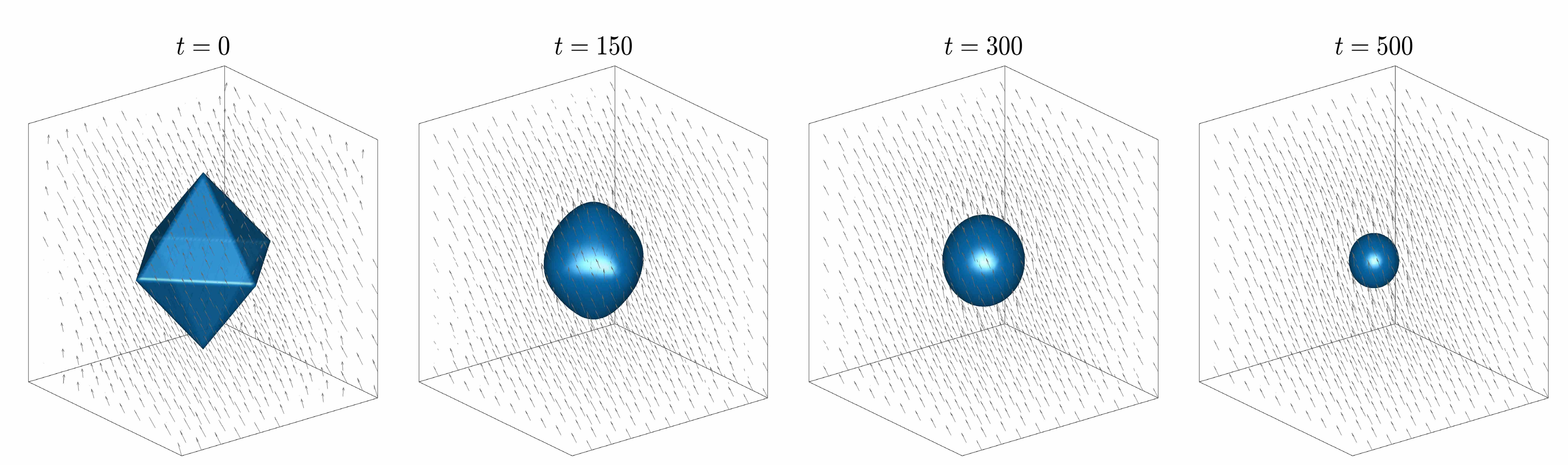}}\\[-1mm]
\subfigure[Top view]{\includegraphics[width=0.98\textwidth]{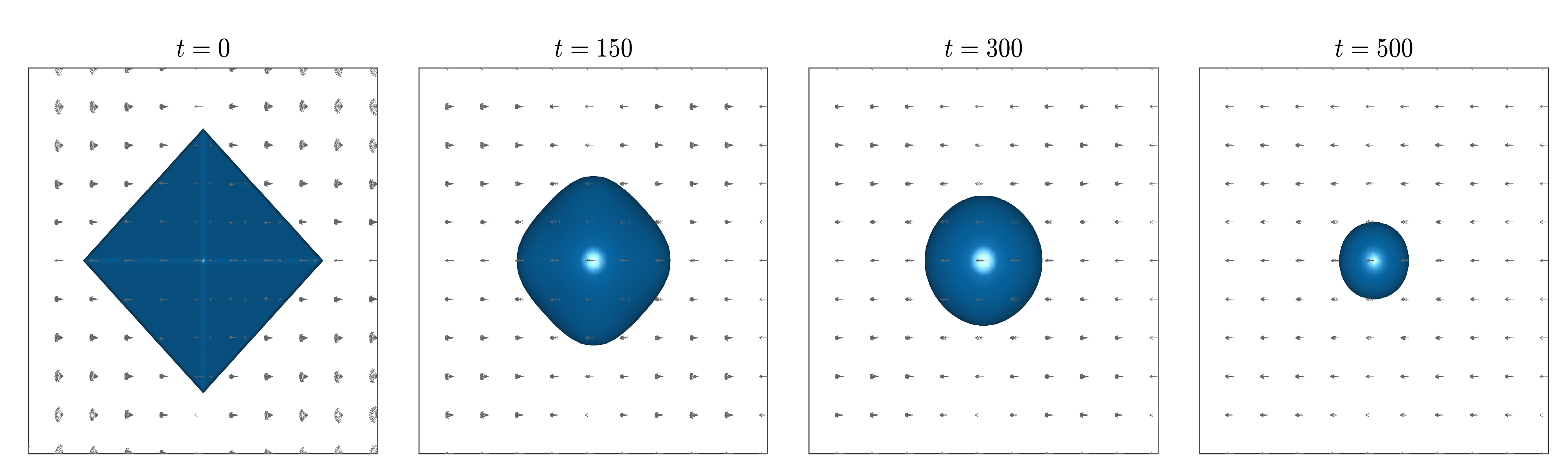}}
\caption{Three-dimensional evolution of the zero-level surface $\det(U)=0$ for the matrix-valued Allen--Cahn equation.}
\label{fig:4.42}
\end{figure}

\section{Conclusion}\label{sec5}
This paper introduces a generalized matrix-valued Allen--Cahn model with the unknown matrix-valued function $U\in\mathbb{R}^{m_1\times m_2}$ and $m_1\geq m_2$. By taking different values of $m_1$ and $m_2$, this model unifies three classical cases: the scalar-valued, vector-valued, and square-matrix-valued Allen--Cahn equations. At the continuous level, we establish the existence and uniqueness of the solution for the model and prove that the solution satisfies the maximum principle and energy dissipation law. At the discrete level, we construct arbitrarily high-order ETDRK schemes. Theoretical analysis and numerical experiments show that the schemes unconditionally preserve the maximum principle. Moreover, the first- and second-order ETDRK$r$ schemes unconditionally satisfy the original energy dissipation law, while third- and higher-order schemes satisfy the property under suitable time-step constraints. 

Although the present work focuses on the orthogonality-promoting potential, the proposed strategy may be extended to a broader class of smooth gradient-flow potentials \cite{kim2025maximum}. A direct extension of the current analysis would require suitable conditions ensuring the invariance of the admissible Frobenius ball, the boundedness and local Lipschitz continuity of the modified nonlinearity, and an appropriate estimate for the energy density. These conditions are sufficient for the present proof and are not claimed to be necessary or optimal. Identifying weaker assumptions and treating more general potentials systematically are left for future work.

Moreover, this generalized matrix-valued Allen--Cahn model demonstrates application potential in the fields of crystallography and materials science, particularly for simulating grain boundary dynamics, crystal growth, and phase transformation processes. Its matrix-valued characteristics enable the capture of complex phenomena such as grain orientation evolution. Future research may explore extending its application to simulate microstructural evolution in crystalline materials and liquid crystal systems, thereby enhancing our understanding of orientation-dependent behavior.

\section*{Acknowledgement}
C. Quan was partially supported by the National Natural Science Foundation of China (Grant No. 12271241), Guangdong Basic and Applied Basic Research Foundation (Grant No. 2023B1515020030), Shenzhen Science and Technology Innovation Program (Grant No. JCYJ20230807092402004), and Hetao Shenzhen-Hong Kong Science and Technology Innovation Cooperation Zone Project (No.HZQSWS-KCCYB-2024016). D. Wang was partially supported by National Natural Science Foundation of China (Grant No. 12422116), Guangdong Basic and Applied Basic Research Foundation (Grant No. 2023A1515012199), Shenzhen Science and Technology Innovation Program (Grant No. JCYJ20220530143803007, RCYX20221008092843046), Hetao Shenzhen-Hong Kong Science and Technology Innovation Cooperation Zone Project (No.HZQSWS-KCCYB-2024016), and the Shenzhen Loop Area Institute (Grant No. FPF10120250014 and Contract No. SLAI2026020007).

\bibliographystyle{amsplain}
\bibliography{refs}

@article{brutman1997lebesgue,
  author  = {Brutman, L.},
  title   = {Lebesgue Functions for Polynomial Interpolation: A Survey},
  journal = { Ann. Numer. Math.},
  volume  = {4},
  pages   = {111--127},
  year    = {1997}
}

@book{trefethen2019approximation,
  author    = {Trefethen, L. N.},
  title     = {Approximation Theory and Approximation Practice, Extended Edition},
  publisher = {SIAM},
  year      = {2019}
}

@article{kim2025maximum,
  title={Maximum principle preserving the unconditionally stable method for the {Allen--Cahn} equation with a high-order potential},
  author={Kim, J.},
  journal={Electron. Res. Arch.},
  volume={33},
  number={1},
  pages={433--446},
  year={2025}
}

@article{grasselli2024multi,
  title={Multi-component conserved {Allen--Cahn} equations},
  author={Grasselli, M. and Poiatti, A.},
  journal={Interfaces Free Bound.},
  volume={26},
  number={4},
  pages={489--541},
  year={2024}
}

@article{staublin2022phase,
  title={Phase-field model for anisotropic grain growth},
  author={Staublin, P.  and Mukherjee, A. and Warren, J.A. and Voorhees, P.W.},
  journal={Acta Mater.},
  volume={237},
  pages={118169},
  year={2022}
}

@article{wang2019interface,
  title={Interface dynamics for an {Allen--Cahn-type} equation governing a matrix-valued field},
  author={Wang, D. and Osting, B. and Wang, X.-P.},
  journal={Multiscale Model. Simul.},
  volume={17},
  number={4},
  pages={1252--1273},
  year={2019}
}

@article{osting2020diffusion,
  author  = {Osting, B. and Wang, D.},
  title   = {A diffusion generated method for orthogonal matrix-valued fields},
  journal = {Math. Comput.}, 
  volume  = {89 (322)},
  year    = {2020},
  pages   = {515-550}
}

@article{fei2023matrix,
  title={Matrix-valued {Allen--Cahn equation and the Keller--Rubinstein--Sternberg} problem},
  author={Fei, M. and Lin, F. and Wang, W. and Zhang, Z.},
  journal={Invent. math.},
  pages={1--80},
  year={2023}
}

@article{sun2023maximum,
  title={Maximum bound principle for matrix-valued {Allen-Cahn} equation and integrating factor {Runge-Kutta} method},
  author={Sun, Y. and Zhou, Q.},
  journal={Numer. Algorithms},
  pages={1--39},
  year={2023},
  publisher={Springer}
}

@article{du2021maximum,
  author  = {Du, Q. and Ju, L. and Li, X. and Qiao, Z.},
  title   = {Maximum bound principles for a class of semilinear parabolic equations and exponential time differencing schemes},
  journal = {SIAM Rev.}, 
  volume  = {63(2)},
  year    = {2021},
  pages   = {317-359}
}

@article{li2022stability2,
  title={Stability and convergence of {Strang splitting}. {Part II}: tensorial {Allen-Cahn} equations},
  author={Li, D. and Quan, C. and Xu, J.},
  journal={J. Comput. Phys.},
  volume={454},
  pages={110985},
  year={2022}
}

@article{liu2025maximum,
  title={On the maximum bound principle and energy dissipation of exponential time differencing methods for the matrix-valued {Allen--Cahn} equation},
  author={Liu, Y. and Quan, C. and Wang, D.},
  journal={IMA J. Numer. Anal.},
  volume={45},
  number={6},
  pages={3342--3377},
  year={2025}
}

@article{quan2026unconditional,
  title={Unconditional energy dissipation of Strang splitting for the matrix-valued {Allen--Cahn} equation},
  author={Quan, C. and Tang, T. and Wang, D.},
  journal={J. Differ. Equations},
  volume={453},
  pages={113825},
  year={2026}
}

@article{batard2014covariant,
  title={On covariant derivatives and their applications to image regularization},
  author={Batard, T. and Bertalm{\'\i}o, M.},
  journal={SIAM J. Imaging Sci.},
  volume={7},
  number={4},
  pages={2393--2422},
  year={2014}
}

@article{rosman2014augmented,
  title={{Augmented-Lagrangian} regularization of matrix-valued maps},
  author={Rosman, G. and Tai, X.-C. and Kimmel, R. and Bruckstein, A.M.},
  journal={Methods Appl. Anal.},
  volume={21},
  number={1},
  pages={105--122},
  year={2014}
}

@inproceedings{vaxman2016directional,
  title={Directional field synthesis, design, and processing},
  author={Vaxman, A. and Campen, M. and Diamanti, O. and Panozzo, D. and Bommes, D. and Hildebrandt, K. and Ben-Chen, M.},
  booktitle={Computer graphics forum},
  volume={35},
  number={2},
  pages={545--572},
  year={2016}
}

@article{elsey2013simple,
  title={A simple and efficient scheme for phase field crystal simulation},
  author={Elsey, M. and Wirth, B.},
  journal={ESAIM Math. Model. Numer. Anal.},
  volume={47},
  number={5},
  pages={1413--1432},
  year={2013}
}

@article{elsey2014fast,
  title={Fast automated detection of crystal distortion and crystal defects in polycrystal images},
  author={Elsey, M. and Wirth, B.},
  journal={Multiscale Model. Simul.},
  volume={12},
  number={1},
  pages={1--24},
  year={2014}
}

@article{allen1979microscopic,
  title={A microscopic theory for antiphase boundary motion and its application to antiphase domain coarsening},
  author={Allen, S.M. and Cahn, J.W.},
  journal={Acta Metall.},
  volume={27},
  number={6},
  pages={1085--1095},
  year={1979}
}

@article{yang2006numerical,
  title={Numerical simulations of jet pinching-off and drop formation using an energetic variational phase-field method},
  author={Yang, X. and Feng, J.J. and Liu, C. and Shen, J.},
  journal={J. Comput. Phys.},
  volume={218},
  number={1},
  pages={417--428},
  year={2006}
}

@article{gal2010longtime,
  title={Longtime behavior for a model of homogeneous incompressible two-phase flows},
  author={Gal, C.G. and Grasselli, M.},
  journal={Discrete Contin. Dyn. Syst.},
  volume={28},
  number={1},
  pages={1--39},
  year={2010}
}

@article{Liu2025phase,
  title={Phase transition of parabolic {Ginzburg–Landau} equation with potentials of high‐dimensional wells},
  author={Liu, Y.},
  journal={Commun. Pure Appl. Math.},
  volume={78},
  number={6},
  pages={1199-1247},
  year={2025}
}

@article{lin2012reaction,
  title={Phase transition for potentials of high-dimensional wells},
  author={Lin, F. and Pan, X.‐B. and Wang, C. },
  journal={Commun. Pure Appl. Math.},
  volume={65},
  number={6},
  pages={833--888},
  year={2012}
}

@article{rubinstein1989reaction,
  title={Reaction-diffusion processes and evolution to harmonic maps},
  author={Rubinstein, J. and Sternberg, P. and Keller, J.B.},
  journal={SIAM J. Appl. Math.},
  volume={49},
  number={6},
  pages={1722--1733},
  year={1989}
}

@article{rubinstein1989fast,
  title={Fast reaction, slow diffusion, and curve shortening},
  author={Rubinstein, J. and Sternberg, P. and Keller, J.B.},
  journal={SIAM J. Appl. Math.},
  volume={49},
  number={1},
  pages={116--133},
  year={1989}
}

@article{li2021unconditionally,
  title={Unconditionally maximum bound principle preserving linear schemes for the conservative {Allen--Cahn} equation with nonlocal constraint},
  author={Li, J. and Ju, L. and Cai, Y. and Feng, X.},
  journal={J. Sci. Comput.},
  volume={87},
  pages={1--32},
  year={2021}
}

@article{fu2022energy,
  author  = {Fu, Z. and Yang, J.},
  title   = {Energy-decreasing exponential time differencing {Runge–Kutta} methods for phase-field models},
  journal = {J. Comput. Phys.}, 
  volume  = {454},
  year    = {2022},
  pages   = {110943}
}

@article{quan2025maximum,
    title = {Maximum bound principle and original energy dissipation of arbitrarily high-order {ETD Runge–Kutta} schemes for {Allen–Cahn} equations},
    author = {Quan, C. and Wang, X. and Zheng, P. and Zhou, Z.},
    journal = {IMA J. Numer. Anal.},
    pages = {draf069},
    year = {2025}
}

@article{du1992analysis,
  title={Analysis and approximation of the {Ginzburg--Landau} model of superconductivity},
  author={Du, Q. and Gunzburger, M. and Peterson, J.},
  journal={SIAM Rev.},
  volume={34},
  number={1},
  pages={54--81},
  year={1992}
}
\end{document}